\documentclass{zml-article}

\usepackage[noabbrev,capitalise,english]{cleveref}

\usepackage{amsmath}
\usepackage{amscd}
\usepackage{amssymb}
\usepackage{amsfonts}
\usepackage{mathrsfs}
\usepackage{euscript}
\usepackage{oldgerm}
\usepackage{calligra}
\usepackage{mathrsfs}
\usepackage{bold-extra}
\usepackage{pifont}
\usepackage{hyperref}
\usepackage{url}
\usepackage{mathptmx}

\usepackage{titlesec}
\counterwithin{section}{part}

\newtheorem{theorem}{\textbf{\textsc{Theorem}}}[section]
\newtheorem{definition}[theorem]{\textbf{\textsc{Definition}}}

\newtheorem{convention}[theorem]{\textbf{\textsc{Convention}}}
\newtheorem{lemma}[theorem]{\textbf{\textsc{Lemma}}}
\newtheorem{corollary}[theorem]{\textbf{\textsc{Corollary}}}
\newtheorem{proposition}[theorem]{\textbf{\textsc{Proposition}}}
\newtheorem{example}[theorem]{\textbf{\textsc{Example}}}

\newtheorem{question}[theorem]{\textbf{\textsc{Question}}}

\newtheorem{remark}[theorem]{\textbf{\textsc{Remark}}}

\title{\bf On Chaitin's Heuristic Principle\\ and Halting Probability}

   \shorttitle{\sl On Chaitin's Heuristic Principle}

\author[1]{\sc Saeed  Salehi}

\affil[1]{\sl }

\corrauthor{Saeed Salehi}{root@SaeedSalehi.ir}

\abstract{It would be a heavenly reward if there were a method of weighing theories and sentences in such a way that a theory could never prove a heavier sentence (Chaitin’s Heuristic Principle). Alas, no satisfactory measure has been found so far, and this dream seemed too good ever to come true. In the first part of this paper, we attempt to revive Chaitin’s lost paradise of heuristic principle as much as logic allows. In the second part, which is a joint work with M. Jalilvand and B. Nikzad, we study Chaitin’s well-known constant Omega and show that this number is not a probability of halting the randomly chosen input-free programs under any infinite discrete measure. We suggest several methods for defining halting probabilities using various measures.

\medskip


 \medskip

\noindent
{\em 2020 Math Subject Classification}.
03F40,   	
68Q30,
60A10, 28A05, 68Q04, 03D10.
}

\begin{document}

\maketitle

\part{Weighing Theories: On {Chaitin}'s Heuristic Principle}

\section{Introduction and Preliminaries}\label{sec:intro}
The history of the heuristic principle goes back to at least  1974,  when {Chaitin} wrote in \cite{Chaitin74} that ``there are circumstances in which [$\dots$] it is possible to reason in the following manner.
If a set of theorems constitutes $t$ bits of information,
and a set of axioms contains less than $t$ bits of information, then it is impossible to deduce these theorems from these axioms'' (p.~404).
Let us put this formally as

\begin{definition}[Heuristic Principle, {\sf HP}]\label{def:hp}
\noindent

\indent
A real-valued mapping $\mathcal{W}$ on theories and sentences \textup{(}over a fixed language\textup{)} is said to satisfy {Chaitin}'s heuristic principle, {\sf HP}, when for every theory $T$ and every sentence $\psi$, if we have  $\mathcal{W}(\psi)\!>\!\mathcal{W}(T)$ then $T$ cannot prove  $\psi$; in symbols

\qquad \qquad \qquad
\textup{(}{\sf HP}\textup{)} \quad $\mathcal{W}(\psi)\!>\!\mathcal{W}(T) \;\Longrightarrow\; T\nvdash\psi$.
\hfill \ding{71}
\end{definition}

An equivalent formulation is  ``{the theorems weigh no more than the  theory}'':
$T\vdash\psi \,\Rightarrow\, \mathcal{W}
(T)\geqslant\mathcal{W}(\psi)$.
The principle was suggested after {Chaitin}'s (proof of {G\"odel}'s first) incompleteness theorem, which says that for every sufficiently strong and consistent theory $T$, there exists a constant $c$ such that for no $w$, $T$  can  prove that ``the {Kolmogorov} complexity of $w$ is greater than $c$''
 (below, we will see a formal  definition of  {Kolmogorov} complexity).
It was criticized by several authors, probably  the first time in 1989 by {van Lambalgen}, who wrote in \cite{vanLamb89} that ``Chaitin's views enjoy a certain popularity'' (p.~1389) and that ``Chaitin's mathematics does not support his philosophical
conclusions'' (p.~1390).
He
concluded ``that the {\em complexity} of the axioms is not a good measure of information'' (p.~1395; original italics).

{Chaitin} retreated, taking just one step back and not completely, in one of his 1992 papers \cite{Chaitin92}:
``In fact, {\em any} set of axioms that yields an infinite
set of theorems {\em must} yield
theorems with arbitrarily high complexity!
[$\dots$]
So what is to become of our heuristic principle that `A set of axioms of
complexity $N$ cannot yield a theorem of complexity substantially greater than
$N$'??? An improved version of this heuristic principle, which is not really any
less powerful than the original one, is this: `One cannot prove a theorem
from a set of axioms that is of greater complexity than the axioms {\em and know}
that one has done this. I.e., one {\em cannot realize}  that a theorem is of
substantially greater complexity than the axioms from which it has been
deduced, if this should happen to be the case''' (p.~115, original italics).
He then tries
``to avoid all these problems and discussions by
rephrasing [his] fundamental principle in the following totally unobjectionable
form: `A set of axioms of complexity $N$ cannot yield a theorem that asserts
that a specific object is of complexity substantially greater than $N$.' It was
removing the words `asserts that a specific object' that yielded the slightly
overly-simplified version of the principle'' (p.~116).

The fact of the matter is that {Chaitin}'s last statement, that ``a set of axioms of
complexity $N$ cannot yield a theorem that asserts that a specific object is of complexity
substantially greater than $N$'', is nothing but a reformulation of his incompleteness theorem, and the mildly diluted statement, that ``one cannot prove a theorem from a set of axioms that is of greater complexity
than the axioms {\em and know} that [that]   theorem
is of substantially greater complexity than the axioms from which it has been deduced'', follows from his incompleteness theorem and is indeed far from the original {\sf HP}.

The criticism continued,  for good reasons. In 1996,  {Fallis} \cite{Fallis96} noted that ``for any sound formal system
$FS$, there are infinitely many formulas which have greater complexity than
$FS$ and are provable in $FS$.
[$\dots$] Since only finitely many of the (infinitely many)
formulas provable in $FS$ have a complexity less than the complexity of $FS$,
there are infinitely many formulas provable in $FS$ with greater complexity
than $FS$'' (p.~265). As a result, Chaitin's claim (that ``if a theorem contains more information than a given set of
axioms, then it is impossible for the theorem to be derived from the axioms'' \cite[p.~264]{Fallis96}) is false.

In 1998, {Raatikainen}, maybe unaware of  \cite{Fallis96} as he did not cite it, wrote in
\cite{Raatikainen98}
 that ``Chaitin's metaphor that `if one has ten pounds of axioms and a twenty-pound theorem, then that theorem cannot be derived from those axioms', if referring to Chaitin's theorem, seems to commit [the] confusion [that] it compares the complexity of axioms as {\em mentioned} and the complexity asserted by a theorem when {\em used}'' (p.~581; emphasis added).

Anyhow,  {\sf HP} was too beautiful a dream to let go easily. In 2004, {Sj\"ogren} designed, in his licentiate thesis  \cite{Sjorgen04}, ``a measure of the power'' of theories and sentences that satisfy {\sf HP}, where theories  extended Peano's Arithmetic and the sentences  belonged to a rather narrow class of  arithmetical translations of ``the letterless modal sentences'' of G\"odel-L\"ob logic. A report of the results was presented later in a 2008 paper \cite{Sjorgen08}.

In 2005, {Calude} and {J\"urgensen} claimed   in \cite{Caludetal05} that they ``prove that the `heuristic principle' proposed by Chaitin [$\dots$] is correct if we measure the complexity of a string by the difference
between the program-size complexity and the length of the string, [their] $\delta$-complexity'' (pp.~3--4).
The $\delta$-complexity was defined as $\delta(x)=\mathscr{K}(x)\!-\!|x|$, where $|x|$ denotes the length of $x$ and $\mathscr{K}(x)$, the {Kolmogorov} complexity of $x$, is the length of the shortest input-free program (in a fixed programming language) that outputs only $x$ (and then halts).
The main result of \cite{Caludetal05} reads: ``Consider a finitely-specified, arithmetically sound (i.e. each arithmetical
proven sentence is true), consistent theory strong enough to formalize arithmetic, and denote
by ${T}$ its set of theorems [$\dots$] Let $g$ be a G\"odel numbering
for ${T}$. Then, there exists a constant $N$  [$\dots$]  such that $T$
contains no $x$ with $\delta_g(x)>N$'' (p.~9, Theorem 4.6).
This claim was praised by, e.g., {Grenet} (in 2010), who wrote in  \cite{Grenet10} that
``In [2] Chaitin's heuristic
principle is proved to be valid for an appropriate measure [that]  gives us some indication about the reasons certain
statements are unprovable'' (p.~404), and that his study led him
``to modify the definition of $\delta_g$
in order to correct some of the proofs''  (p.~423).

Unfortunately, $\delta$ does not satisfy {\sf HP}, as can be seen by the following argument. Let $\bot$ denote a contradiction, like $p\!\wedge\!\neg p$ or $\exists x(x\!\neq\!x)$. There are two fixed, and rather small, natural numbers $m$ and $n$, such that for every formula $\varphi$ we have $|\bot\!\rightarrow\!\varphi|=|\varphi|+m$ and $\mathscr{K}(\varphi)\leqslant \mathscr{K}(\bot\!\rightarrow\!\varphi)+n$.
For the latter, notice that one can make some small  changes to the shortest input-free program that outputs only $\bot\!\rightarrow\!\varphi$, to get an input-free  program, not necessarily with the shortest length, that outputs only $\varphi$; and those changes are uniform and do not depend on $\varphi$.\footnote{Actually, for most of the standard formalisms and frameworks, it suffices to take $m=2$ and $n=0$.} Now, fix an arbitrary theory $T$ and  assume that $\delta(T)=t$; one could take $t$ to be the constant $N$ in the above quoted Theorem~4.6 of \cite{Caludetal05}. Also, fix a sentence $\textgoth{S}$ with $\delta(\textgoth{S})>t+m+n$ (which should exist by \cite[Corr.~4.2, p.~6]{Caludetal05}). We have

\begin{tabular}{rclcl}
  $\delta(\bot\!\rightarrow\!\textgoth{S})$ & $=$ & $\mathscr{K}(\bot\!\rightarrow\!\textgoth{S})- |\bot\!\rightarrow\!\textgoth{S}|$  & & by the definition of $\delta$, \\
    & $\geqslant$ & $\mathscr{K}(\textgoth{S})-n-(|\textgoth{S}|+m)$   & &   by the choice of $m,n$, \\
   & $=$  & $\delta(\textgoth{S})-(m+n)$ & &  by the definition of $\delta$, \\
    &  $>$ & $t$ & &  by the choice of $\textgoth{S}$,  \\
  &  $=$ & $\delta(T)$ & &  by the definition of $t$.  \\
\end{tabular}

So, by \cite[Thm.~4.6]{Caludetal05}, quoted above,
$T\nvdash(\bot\!\rightarrow\!\textgoth{S})$ since $\delta(\bot\!\rightarrow\!\textgoth{S})\!>\!\delta(T)$; but $\bot\!\rightarrow\!\textgoth{S}$, for every sentence $\textgoth{S}$, is a tautology, and so should be provable in every theory. One can give a similar argument by using the formula $\textgoth{S}\!\rightarrow\!\top$, where $\top$ denotes a tautology, like $p\!\vee\!\neg p$ or $\forall x(x\!=\!x)$; or by using the tautologies  $p\!\rightarrow\!(\textgoth{S}\!\rightarrow\!p)$
or $\neg p\!\rightarrow\!(p\!\rightarrow\!\textgoth{S})$
or $[(p\!\rightarrow\!\textgoth{S})
\!\rightarrow\!p]\!\rightarrow\!p$, for a fixed, short, and uncomplicated formula $p$.\footnote{It would be a good exercise to go through the arguments of \cite{Caludetal05} and \cite{Grenet10}, pinpoint the possible errors, and see what went wrong.}
 Summing up, the following holds  according to the cited results and our argument  above:
\begin{proposition}[{\sf HP} does not hold so far]\label{prop:nohp}
\noindent

\indent
Neither the {Kolmogorov}-complexity \textup{(the length of the shortest input-free program that outputs only the axioms, or the theorems, of the theory)} nor the $\delta$-complexity \textup{(the difference between the {Kolmogorov}-complexity and the length)}  satisfies {\sf HP}.
\hfill\ding{113}
\end{proposition}

This is not the end of the story. In 2021, {Porter} \cite{Porter21} hoped to show   ``a possible vindication of Chaitin's interpretation [{\sf HP}],
drawing upon the [$\dots$] recent work of Bienvenu et al.\ \cite{BRSTV14} that extends and refines [Chaitin's incompleteness theorem]'' (p.~149). He was content  with a weak version of {\sf HP}, as his ``statement is significantly weaker than being a full-fledged instance of
Chaitin's heuristic principle'' (p.~160).

In the rest of this paper, we suggest some ways of weighing theories and sentences that satisfy {\sf HP}.

\section{Weighing Theories}\label{sec:wei}
We can work with  {\em theories} only and dismiss the {\em sentences}; it   suffices to take the weight of a sentence $\psi$ to be the weight of the theory $\{\psi\}$. For us, a {\em theory} is an arbitrary {\em set of sentences} (we will consider {\sc re} theories later in \S\S  \ref{sec:re}).
The most trivial mapping that satisfies {\sf HP} is the constant weight (all the theories have a fixed weight): since $\mathcal{W}(U)>\mathcal{W}(T)$ can never hold for any theories $T$ and $U$, then ``$\mathcal{W}(U)>\mathcal{W}(T)\,\Rightarrow\,T\nvdash U$'' is vacuously true. In the sequel, we note that {\sf HP} can be satisfied by some other, less trivial weights.

\subsection{\textsf{HP}--Satisfying Weights}
{\sf HP} forces the existence of a minimum and a maximum.

\begin{theorem}[$\mathscr{W}$ has a min and a max if it satisfies {\sf HP}]\label{minmax}
\noindent

\indent
Every weighting $\mathscr{W}$ that satisfies {\sf HP} has a minimum and a maximum.
\end{theorem}
\begin{proof}
\noindent

\indent
For a theory $T$, we have $\bot\vdash T\vdash\top$, where we recall that $\bot$ is a contradiction and $\top$ is a tautology. Now, by {\sf HP}, we have $\mathscr{W}(\bot)\geqslant\mathscr{W}(T)
\geqslant\mathscr{W}(\top)$. Thus, $\min_{\mathscr{W}}=\mathscr{W}(\top)$ and $\max_{\mathscr{W}}=\mathscr{W}(\bot)$.
\end{proof}

\begin{corollary}[integer-valued {\sf HP}-satisfying weights are finitely many valued]\label{corollary1}
\noindent

\indent
If an integer-valued weighing, like the Kolmogorov complexity $\mathscr{K}$ or its difference with the length (the $\delta$ complexity), satisfies {\sf HP}, then it can take finitely many values only.
\hfill\ding{113}
\end{corollary}

We will see in Theorem~\ref{thm:noep} below that there are plenty of finitely many valued weightings that satisfy {\sf HP}.
We already noted in Proposition~\ref{prop:nohp} that neither $\mathscr{K}$ nor $\delta$ satisfies {\sf HP}.
Now, we notice a couple of properties of arithmetical theories that extend Robinson's Arithmetic $\textit{\textbf{Q}}$, which imply that finitely many weights are not suitable for these theories, even if they satisfy {\sf HP}.

\begin{theorem}[Arithmetical theories form an infinite and a pseudo-dense  hierarchy]\label{thm:arith}
\noindent

\begin{enumerate}
  \item There is an infinite hierarchy of finitely axiomatizable  arithmetical theories containing $\textit{\textbf{Q}}$ with strictly increasing proof power.
  \item Between every two finite extensions of $\textit{\textbf{Q}}$, one of which is a strict sub-theory of the other, there exists another finite theory that sits strictly between the two. In other words, if $\textit{\textbf{Q}}\subseteq S\subsetneq T$ for two finite theories $S$ and $T$, then there is a finite theory $U$ such that $S\subsetneq U\subsetneq T$.
\end{enumerate}
\end{theorem}
\begin{proof}
\noindent

\indent
(1):  For a consistent theory $T$ that contains  $\textit{\textbf{Q}}$, by G\"odel-Rosser's Incompleteness Theorem, there exists a sentence $\rho$, called a Rosser Sentence of $T$, such that $T$ can neither prove nor disprove it. Therefore, both theories $T_0=T+\rho$ and $T_1=T+\neg\rho$ are consistent.
Now, $T_0$ and $T_1$ strictly contain $T$ (and $\textit{\textbf{Q}}$), and they are both finite if $T$ is so. Continuing this way, one can show the existence of an infinite (strictly increasing in power) hierarchy of finite arithmetical theories.

\indent
(2): Let us identify a finitely axiomatizable theory with a single sentence that axiomatizes it (which can be the conjunction of the finitely many axioms of the theory). Let $S$ and $T$ be two sentences such that they both contain $\textit{\textbf{Q}}$ (i.e., $S,T\vdash\textit{\textbf{Q}}$) and $S$ is a strict subtheory of $T$ (i.e., $T\vdash S$ and $S\nvdash T$). So, the theory $W=S+\neg T$ is consistent; it also contains $\textit{\textbf{Q}}$.
By G\"odel-Rosser's Incompleteness Theorem, there exists a sentence $\rho$ independent from $W$. So, $W\nvdash\rho$ and $W\nvdash\neg\rho$. Let $U=T\vee (S\wedge\rho)$. This sentence contains $\textit{\textbf{Q}}$, as $U\vdash\textit{\textbf{Q}}$ follows from $T\vdash\textit{\textbf{Q}}$ and $S\vdash\textit{\textbf{Q}}$ (with $S\wedge\rho\vdash S$). Obviously, $U$ lies between $S$ and $T$, since $T\vdash U$ is clear from the tautology $T\rightarrow T\vee X$, and $U\vdash S$ follows from the assumption $T\vdash S$. We show that $U$ strictly lies between $S$ and $T$, i.e., $U\nvdash T$ and $S\nvdash U$.
The former follows from $W\nvdash\neg\rho$, which implies $S\wedge\neg T\nvdash\neg\rho$, or equivalently $S\wedge\rho\nvdash T$, thus $U\nvdash T$. The latter follows from $W\nvdash\rho$, which implies $S\wedge\neg T\nvdash\rho$, thus $S\nvdash T\vee\rho$, so $S\nvdash U$. Essentially the same argument shows that $U'=T\vee(S\wedge\neg\rho)$ too strictly lies between $S$ and $T$. It can also be shown that $U$ and $U'$ are incomparable with each other, i.e., $U\nvdash U'$ and $U'\nvdash U$.
\end{proof}

Consider the finite theory $\textit{\textbf{Q}}$, and let $\rho$ be a Rosser sentence of it. So, $\textit{\textbf{Q}}\subsetneq\textit{\textbf{Q}}+\rho$. By Theorem~\ref{thm:arith}, there are finite and consistent theories $\{T_m\}_m$, where $m$ ranges over integer numbers, such that
$$\textit{\textbf{Q}}\subsetneq\cdots\subsetneq
T_{-2}\subsetneq T_{-1}\subsetneq T_0\subsetneq
T_1\subsetneq T_2\subsetneq\cdots\subsetneq\textit{\textbf{Q}}+\rho.$$

So, an {\sf HP}-satisfying weight should contain a 
decreasing and 
an increasing sequence of values.
Thus, as a corollary of Corollary~\ref{corollary1}, we get the following.

\begin{corollary}[integer-valued {\sf HP}-satisfying weights for arithmetical theories]\label{corollary2}
\noindent

\indent
If an integer-valued weighing, like the Kolmogorov complexity $\mathscr{K}$ or its difference with the length (the $\delta$ complexity), satisfies {\sf HP}, then  it will have a fixed value for infinitely many distinct arithmetical theories.
\hfill\ding{113}
\end{corollary}

\subsection{Finitely Many Weights}
Let $\nu$ be a propositional evaluation from formulas to $\{0,1\}$, where $0$ indicates the \textsf{falsum} and $1$ the \textsf{truth}. If $P\rightarrow Q$ holds, then we have $\nu(P)\leqslant\nu(Q)$. This suggests the following:

\begin{definition}[$\mathcal{W}_V,
\mathcal{W}_{\mathfrak{M}}$]
\label{def:mu}
\noindent

\indent
Let $\nu$ be a mapping from propositional atoms to $\{\textsf{false},\textsf{true}\}$, and let $V$ be its truth-table extension to all the propositional formulas. For a formula $\varphi$, let $V\vDash\varphi$ mean that $V(\varphi)=\textsf{true}$; and for a theory $T$ let $V\vDash T$ mean that $V\vDash\tau$ holds for each and every element (axiom) $\tau$ of $T$. Let $\mathcal{W}_V$ be the following mapping, where $T$ is a propositional theory.
$$\mathcal{W}_V(T)=\begin{cases}
                     0 & \mbox{if } V\vDash T,  \\
                     1 & \mbox{if } V\nvDash T.
                   \end{cases}$$
Likewise, for a fixed first-order structure $\mathfrak{M}$, let $\mathcal{W}_{\mathfrak{M}}$ be the following mapping:
$$\mathcal{W}_{\mathfrak{M}}(T)=\begin{cases}
                     0 & \mbox{if } \mathfrak{M}\vDash T,  \\
                     1 & \mbox{if } \mathfrak{M}\nvDash T;
                   \end{cases}$$
where $T$ is a  first-order theory.
\hfill \ding{71}
\end{definition}

\begin{theorem}[$\mathcal{W}_V,\mathcal{W}_{\mathfrak{M}}$ satisfy {\sf HP}]\label{thm:m}
\noindent

\indent
For every evaluation $V$ and every structure $\mathfrak{M}$,   
$\mathcal{W}_V$ and $\mathcal{W}_{\mathfrak{M}}$ satisfy {\sf HP}.
\end{theorem}
\begin{proof}
\noindent

\indent
If $\mathcal{W}_{\mathfrak{M}}(U)>\mathcal{W}_{\mathfrak{M}}(T)$, then $\mathcal{W}_{\mathfrak{M}}(U)=1$ and
$\mathcal{W}_{\mathfrak{M}}(T)=0$, so $\mathfrak{M}\vDash T$ but $\mathfrak{M}\nvDash U$, therefore    $T\nvdash U$.
\end{proof}

Let us recall that no structure can satisfy an {\em inconsistent} theory $T$ (for which we write  $T\vdash\bot$); and every structure  satisfies a {\em tautological} theory $T$ (written as $\top\vdash T$). If we replace ``$\mathfrak{M}\vDash$'' with ``$\top\vdash$'' in Definition~\ref{def:mu}, then we get the following weighing $\mathcal{W}_{\top}$; the other weighing $\mathcal{W}^{\bot}$ is its dual.

\begin{definition}[$\mathcal{W}_{\top},
\mathcal{W}^{\bot}$]\label{def:botop}
\noindent

\indent
Let $\mathcal{W}_{\top}$ and $\mathcal{W}^{\bot}$ be the following mappings,
$$\mathcal{W}_{\top}(T)=\begin{cases}
                          0 & \mbox{if } \top\vdash T, \\
                          1 & \mbox{if } \top\nvdash T;
                        \end{cases}\quad\textrm{ and }\quad
\mathcal{W}^{\bot}(T)=\begin{cases}
                          0 & \mbox{if } T\nvdash\bot, \\
                          1 & \mbox{if } T\vdash\bot;
                           \end{cases}$$
for a theory $T$.
\hfill \ding{71}
\end{definition}
The weighing $\mathcal{W}^{\bot}$  is the so-called ``drastic inconsistency measure'', introduced  by {Hunter} and {Konieczny} in 2008; see e.g.\ \cite[Def.~5, p.~1011]{HK10}.
It is easy to see that both the mappings  $\mathcal{W}_{\top}$ and $\mathcal{W}^{\bot}$ satisfy {\sf HP} (see Theorem~\ref{thm:wv} below). In fact,  $\top$ and $\bot$ play no special roles  in $\mathcal{W}_{\top}$ or $\mathcal{W}^{\bot}$.

\begin{definition}[$\mathcal{W}_{\mathbb{V}},
\mathcal{W}^{\mathbb{V}}$]\label{def:wv}
\noindent

\indent
For a fixed theory $\mathbb{V}$, let $\mathcal{W}_{\mathbb{V}}$ and $\mathcal{W}^{\mathbb{V}}$ be the following mappings,
$$\mathcal{W}_{\mathbb{V}}(T)=\begin{cases}
                          0 & \mbox{if } \mathbb{V}\vdash T, \\
                          1 & \mbox{if } \mathbb{V}\nvdash T;
                        \end{cases}\quad\textrm{ and }\quad
\mathcal{W}^{\mathbb{V}}(T)=\begin{cases}
                          0 & \mbox{if } T\nvdash\mathbb{V}, \\
                          1 & \mbox{if } T\vdash\mathbb{V};
                           \end{cases}$$
where $T$ is a theory.
\hfill \ding{71}
\end{definition}
Below, we will show that both $\mathcal{W}_{\mathbb{V}}$ and $\mathcal{W}^{\mathbb{V}}$ satisfy {\sf HP}. Let us note that

(a) If $\mathbb{V}$ is tautological, then $\mathcal{W}_{\mathbb{V}}$  is $\mathcal{W}_{\top}$ in Definition~\ref{def:botop},  and $\mathcal{W}^{\mathbb{V}}$ is the constant weighing $1$; and

(b) If $\mathbb{V}$ is inconsistent, then $\mathcal{W}_{\mathbb{V}}$ is the constant weighing $0$, and $\mathcal{W}^{\mathbb{V}}$ is $\mathcal{W}^{\bot}$ in Definition~\ref{def:botop}.

\begin{theorem}[$\mathcal{W}_{\mathbb{V}},\mathcal{W}^{\mathbb{V}}$ satisfy {\sf HP}]\label{thm:wv}
\noindent

\indent
For a fixed theory $\mathbb{V}$,  both $\mathcal{W}_{\mathbb{V}}$ and $\mathcal{W}^{\mathbb{V}}$ satisfy {\sf HP}.
\end{theorem}
\begin{proof}
\noindent

\indent
If $\mathcal{W}^{\mathbb{V}}(U)>
\mathcal{W}^{\mathbb{V}}(T)$, then $\mathcal{W}^{\mathbb{V}}(U)=1$ and
$\mathcal{W}^{\mathbb{V}}(T)=0$, so $U\vdash\mathbb{V}$ but $T\nvdash\mathbb{V}$, therefore  $T\nvdash U$. The case of $\mathcal{W}_{\mathbb{V}}$ is very similar to the proof of Theorem~\ref{thm:m} (just replace   ``$\mathfrak{M}\vDash$'' with ``$\mathbb{V}\vdash$'').
\end{proof}

One main tool in the proofs was the transitivity of the deduction relation: if $S\vdash T\vdash U$, then $S\vdash U$.
There are some other {\sf HP}-satisfying mappings that have more than two values. Let us skip the proof of the following proposition, which could be an interesting exercise in elementary logic.

\begin{proposition}[Some {\sf HP}-satisfying weightings with  more than two  values]\label{prop:3,5}
\noindent

\indent
The following mapping, for a theory $T$, satisfies {\sf HP}.
$$T\mapsto\begin{cases}
            0 & \mbox{if } \top\vdash T \textrm{ (i.e., if $T$ is tautological)}; \\
            1 & \mbox{if } \top\nvdash T\nvdash\bot \textrm{ (i.e., if $T$ is non-tautological and consistent)}; \\
            2 & \mbox{if } T\vdash\bot \textrm{ (i.e., if $T$ is inconsistent)}.
          \end{cases}$$
Fix a consistent and non-tautological theory $\mathbb{V}$ (that is,  $\top\nvdash\mathbb{V}\nvdash\bot$). The following mappings, for a theory $T$, satisfy {\sf HP}.
$$T\mapsto\begin{cases}
            0 & \mbox{if } T\nvdash\mathbb{V}\vdash T;  \\
            1 & \mbox{if } T\vdash\mathbb{V}\vdash T \mbox{ or } T\nvdash\mathbb{V}\nvdash T; \\
            2 & \mbox{if } T\vdash\mathbb{V}\nvdash T \mbox{ and } T\nvdash\bot; \\
            3 & \mbox{if } T\vdash\bot.
          \end{cases}
\,  \textrm{ and }  \,
T\mapsto\begin{cases}
            0 & \mbox{if } \top\vdash T;  \\
            1 & \mbox{if } \top\nvdash T  \mbox{ and } T\nvdash\mathbb{V}\vdash T;\\
            2 & \mbox{if } T\vdash\mathbb{V}\vdash T \mbox{ or } T\nvdash\mathbb{V}\nvdash T; \\
            3 & \mbox{if } T\vdash\mathbb{V}\nvdash T \mbox{ and } T\nvdash\bot; \\
            4 & \mbox{if } T\vdash\bot.
          \end{cases}$$
\hfill \ding{113}
\end{proposition}

\section{The Equivalence Principle}
The converse of {\sf HP}, that is,
$${\sf HP}^{-1}\!: \quad  T\nvdash U\Rightarrow
\mathcal{W}(U)\!>\!\mathcal{W}(T),  \textrm{ for  theories }  T \textrm{ and } U,$$ cannot hold for real-valued weightings (see also \cite[pp.~184 \& 198]{Sjorgen08}). The reason is that, firstly,  ${\sf HP}^{-1}$ is equivalent to
$\mathcal{W}(U)\!\leqslant\!\mathcal{W}(T)
\Rightarrow T\vdash U,$
and, secondly, there are {\em incomparable} theories (neither of which can prove the other).
In fact, every non-provable and non-refutable sentence is incomparable to its negation; take, for example, any atom in propositional logic or $\forall x\forall y(x\!=\!y)$ in predicate logic with equality. For
incomparable theories $T$ and $T'$ and a real-valued weighing $\mathcal{W}$,  either $\mathcal{W}(T)\!\leqslant\!\mathcal{W}(T')$   or $\mathcal{W}(T')\!\leqslant\!\mathcal{W}(T)$ holds, but neither $T'\vdash T$ holds nor $T\vdash T'$.
So, ${\sf HP}^{-1}$ is out of the question as long as our weighing mappings are 
real-valued.

Let us now consider a couple of non-real-valued mappings that satisfy both {\sf HP} and ${\sf HP}^{-1}$. For the first example,   consider the  deductive closure $T^{\vdash}$ of a theory $T$, which consists of all the $T$-provable sentences.
Now, for all theories $T$ and $U$, we have $T\vdash U\!\iff\!U\subseteq T\!\iff\!T^{\vdash}\supseteq U^{\vdash}.$
Thus, deductively closed sets can {\em weigh} theories, and they satisfy ${\sf HP + HP}^{-1}$ with the inclusion order ($\supseteq$), which is transitive but not linear.
Our second example will give rise to a real-valued weighing.
\begin{definition}[$\langle \boldsymbol\psi_n\rangle_{n>0}$]\label{def:sn}
\noindent

\indent
Fix $\boldsymbol\psi_1, \boldsymbol\psi_2, \boldsymbol\psi_3, \cdots$ to be a list of all the sentences \textup{(}in a fixed countable language and computing framework\textup{)}. The list can be taken to be  effective in the sense that for a given $n>0$ it is possible to find, in a computable way, the sentence $\boldsymbol\psi_n$.
\hfill \ding{71}
\end{definition}
We consider the infinite binary $\{0,1\}$-sequences for our second example.

\begin{definition}[$\sigma,\sqsubseteq$]\label{def:sigma}
\noindent

\indent
For a theory $T$, let $\sigma(T)=\langle\mathcal{W}^{\{\boldsymbol\psi_n\}}(T)
\rangle_{n>0}$ \textup{(}see Definition~\ref{def:wv}\textup{)}.
For   binary sequences $\varsigma=\langle\varsigma_n\rangle_{n>0}$ and $\tau=\langle\tau_n\rangle_{n>0}$, let $\varsigma\sqsubseteq\tau$ mean that $\varsigma_n\leqslant\tau_n$ holds for every $n>0$.
\hfill \ding{71}
\end{definition}

Let us note that the binary relation $\sqsubseteq$ is transitive but non-linear (for example, the sequence $\langle 0,1,1,1,\cdots\rangle$ is $\sqsubseteq$-incomparable with $\langle 1,0,1,1,\cdots\rangle$). However,  $\sigma$  satisfies {\sf HP} and ${\sf HP}^{-1}$ with respect to  $\sqsupseteq$.

\begin{proposition}[${\sf HP\!+\!HP}^{-1}$ for $\sigma$  with $\sqsupseteq$]\label{prop:sigma}
\noindent

\indent
For all  theories $T$ and $U$, we have
$T\vdash U \!\iff\!  \sigma(T)\sqsupseteq \sigma(U).$
\end{proposition}
\begin{proof}
\noindent

\indent
(1) If $T\vdash U$, then for every sentence $\boldsymbol\psi_n$,  $U\vdash\boldsymbol\psi_n$ implies $T\vdash\boldsymbol\psi_n$; thus,   $\mathcal{W}^{\{\boldsymbol\psi_n\}}(T)\geqslant \mathcal{W}^{\{\boldsymbol\psi_n\}}(U)$ for every $n>0$; therefore, $\sigma(T)\sqsupseteq \sigma(U)$. (2) If $T\nvdash U$, then for some sentence $\boldsymbol\psi_m$, we have $U\vdash\boldsymbol\psi_m$ but $T\nvdash\boldsymbol\psi_m$ (one can take $\boldsymbol\psi_m$ to be one of the $T$-unprovable axioms of $U$); thus, by $\mathcal{W}^{\{\boldsymbol\psi_m\}}(T)=0$ and
$\mathcal{W}^{\{\boldsymbol\psi_m\}}(U)=1$, we obtain  $\sigma(T)\not\sqsupseteq \sigma(U)$.
\end{proof}

Before going back to real-valued weights, let us notice another property of the binary sequence  $\sigma(T)$, in Definition~\ref{def:sigma}, which will be needed later.

\begin{lemma}[When $\sigma(T)$ is eventually constant]\label{lem}
\noindent

\indent
For a theory $T$,
the sequence  $\sigma(T)$ is  eventually constant if and only if it is all  $1$ if and only if $T$ is inconsistent.
\end{lemma}
\begin{proof}
\noindent

\indent
Since 
$\langle \boldsymbol\psi_n\rangle_{n>0}$ contains infinitely many tautologies,
  $\sigma(T)$ cannot be eventually $0$.
Clearly, $\sigma(T)$ is all $1$ for an inconsistent theory $T$.
Conversely, if  $\sigma(T)$ is eventually $1$,   then $T$ must be inconsistent, since a consistent theory cannot derive infinitely many contradictions that exist in the list $\langle \boldsymbol\psi_n\rangle_{n>0}$.
\end{proof}

We saw that ${\sf HP}^{-1}$ does not hold for   real-valued mappings. So, let us consider a rather weak consequence of ${\sf HP}^{-1}$, whose fulfillment will save the weights from being trivial or finitely-many-valued.

\begin{definition}[Equivalence Principle, {\sf EP}]\label{def:ep}
\noindent

\indent
A real-valued mapping $\mathcal{W}$ on theories  is said to satisfy the Equivalence Principle, {\sf EP}, when for all  theories $T$ and $U$, if we have  $\mathcal{W}(T)\!=\!\mathcal{W}(U)$ then $T$ is equivalent to $U$ \textup{(}i.e., $T\vdash U$ and $U\vdash T$\textup{)}; in symbols

\qquad \qquad \qquad
\textup{(}{\sf EP}\textup{)} \quad $\mathcal{W}(T)\!=\!\mathcal{W}(U) \;\Longrightarrow\; T\equiv U$.
\hfill \ding{71}
\end{definition}

Thus, under {\sf EP}, non-equivalent theories should have different weights; this was not the case for any of the real-valued mappings that we have considered so far.
Let us notice that the converse of {\sf EP}, that is,
$${\sf EP}^{-1}\!: \  T\equiv U \Longrightarrow \mathcal{W}(T)\!=\!\mathcal{W}(U),$$ is a consequence of {\sf HP} (and as we noted above, {\sf EP} follows from ${\sf HP}^{-1}$).

\begin{remark}[Chaitin's Characteristic Constant]\label{rem0}
\noindent

\indent
The minimum natural number $c$ such that for every $w$, $T\nvdash\mathscr{K}(x)\!>\!c$, is called Chaitin's {\em characteristic constant} of the theory $T$, denoted $\complement_T$.
Its existence follows from Chaitin's proof of the incompleteness theorem for sufficiently strong consistent theories.
This constant satisfies {\sf HP}, since if $\complement_U\!>\!\complement_T$, then for some $w$ we should have $U\vdash\mathscr{K}(w)\!>\!\complement_T$, thus $T\nvdash U$. But neither it nor any integer-valued weighing can satisfy {\sf EP}, since there could be infinitely many distinct theories between two consistent theories; see  Corollary~\ref{corollary2}.
\hfill \ding{71}
\end{remark}

Finally, we now introduce a   real-valued weighing that satisfies both {\sf HP} and {\sf EP}.

\begin{definition}[$\mathscr{V}$]
\label{def:wdash}
\noindent

\indent
For a theory $T$, let $\mathscr{V}(T)=\sum_{n>0}2^{-n}
\mathcal{W}^{\{\boldsymbol\psi_n\}}(T)$.
\hfill \ding{71}
\end{definition}

Thus, $\mathscr{V}(T)=0\!\centerdot\!\sigma(T)$ in base 2 (see Definition~\ref{def:sigma}); recall from Definition~\ref{def:wv} that
$$\mathcal{W}^{\{\boldsymbol\psi_n\}}(T)=
\begin{cases}
                          0 & \mbox{if } T\nvdash\boldsymbol\psi_n, \\
                          1 & \mbox{if } T\vdash\boldsymbol\psi_n.
                           \end{cases}$$

\begin{theorem}[$\mathscr{V}$ satisfies {\sf HP+EP}]\label{thm:wdash}
\noindent

\indent
The mapping $\mathscr{V}$ satisfies both  {\sf HP} and {\sf EP} for all   theories.
\end{theorem}
\begin{proof}
\noindent

\indent
Let $T$ and $U$ be two theories.
({\sf HP}): If   $T\vdash U$, then   Proposition~\ref{prop:sigma} implies
$\mathscr{V}(T)
\geqslant\mathscr{V}(U)$.
({\sf EP}):  If  $\mathscr{V}(T)=\mathscr{V}(U)$, then since by Lemma~\ref{lem} neither $\mathscr{V}(T)$ nor $\mathscr{V}(U)$ can be eventually $0$,
  $\mathcal{W}^{\{\boldsymbol\psi_n\}}(T)=
\mathcal{W}^{\{\boldsymbol\psi_n\}}(U)$ holds for every $n>0$. Again,
Proposition~\ref{prop:sigma} implies   $T\equiv U$.
\end{proof}

\subsection{Computability and Probability}
What is the use of a weighing if it cannot be computed from (a finite specification of) the theory? It is easy to see that our mapping $\mathscr{V}$ in Definition~\ref{def:wdash} is computable when the underlying logic is decidable (like propositional logic or monadic first-order logic) and our theories are finite.
So, over a decidable logic, we do have some computable weightings that satisfy both {\sf HP} and {\sf EP} for finite theories.
But the story changes dramatically when the underlying logic is not decidable.

\begin{theorem}[Undecidability implies incomputability]\label{thm:hpep}
\noindent

\indent
Over an undecidable classical logic, no weighing can be computable if it satisfies both {\sf HP} and {\sf EP}.
\end{theorem}
\begin{proof}
\noindent

\indent
Assume that a computable weighing $\mathcal{W}$ satisfies {\sf HP} and {\sf EP}. Let ${\footnotesize  \mathfrak{C}}=\mathcal{W}(\{\bot\})$. Then, for every sentence $\psi$ we have

\begin{tabular}{rclcl}
  $\vdash\psi$ & $\iff$ & $\{\neg\psi\}\equiv\{\bot\}$  & & by classical logic, \\
    & $\iff$ & $\mathcal{W}(\{\neg\psi\})=\mathcal{W}(\{\bot\})$   & &   by {\sf HP+EP}, \\
   & $\iff $  & $\mathcal{W}(\{\neg\psi\})={\footnotesize  \mathfrak{C}}$ & &  by the definition of ${\footnotesize  \mathfrak{C}}$.
\end{tabular}

Thus, the logic should be decidable by the computability of $\mathcal{W}$ (and of ${\footnotesize  \mathfrak{C}},\neg$).
\end{proof}

So, we cannot have any computable weighing that satisfies both {\sf HP} and {\sf EP} over first-order logic with a binary relation  symbol, even if our theories are all finite.
One may wonder if Theorem~\ref{thm:hpep} still holds if we loosen the requirements to not require {\sf EP}.  We noted at the beginning of \S\ref{sec:wei} that each constant weighing satisfies {\sf HP}; those weightings are  obviously computable. For a decidable structure  $\mathfrak{M}$, the weight $\mathcal{W}_{\mathfrak{M}}$ in Definition~\ref{def:mu} is computable and satisfies {\sf HP} by Theorem~\ref{thm:m}. Let us recall that a structure $\mathfrak{M}$ is decidable when there exists an algorithm that decides (outputs Yes or No) if a given sentence $\psi$ holds in $\mathfrak{M}$ or not (whether $\mathfrak{M}\vDash\psi$ or $\mathfrak{M}\nvDash\psi$).  Every finite structure over a finite language is decidable, such as the ordered set $\{1,2,\dots,n\}$ over the language $\{<\}$, or the ring $\mathbb{Z}_n(=\mathbb{Z}/n\mathbb{Z})$ over the language $\{+,\times\}$.

\begin{theorem}[Computable weightings satisfying {\sf HP} but not {\sf EP}]\label{thm:noep}
\noindent

\indent
 Over a finite language, for every natural number $n$, there is some $(n+1)$-valued weighing of finite theories that is computable and satisfies {\sf HP}.
\end{theorem}
\begin{proof}
\noindent

\indent
Let the set $\mathcal{C}$ consist of $n$ finite, pairwise non-equinumerous structures. For every sentence $\psi$, let $\mathscr{W}(\psi)$ be the number of structures in $\mathcal{C}$ that do {\em not} satisfy $\psi$. For a finite theory $T$, let $\mathscr{W}(T)$ be $\mathscr{W}(\tau)$, where $\tau$ is the conjunction of the finitely many axioms of $T$. By the decidability of the structures in the finite set $\mathcal{C}$, the weighing $\mathscr{W}$ is computable. We show that it satisfies {\sf HP}. Suppose $\mathscr{W}(\psi)>\mathscr{W}(\varphi)$ for sentences $\varphi$ and $\psi$. Then the number of structures in $\mathcal{C}$ that do not satisfy $\psi$ is strictly greater than the number of structures in $\mathcal{C}$ that do not satisfy $\varphi$. By the pigeonhole principle, one of the structures in $\mathcal{C}$ that does not satisfy $\psi$, say $\mathfrak{M}$, should satisfy $\varphi$.  Thus, $\mathfrak{M}\nvDash\psi$ but $\mathfrak{M}\vDash\varphi$; therefore, $\varphi\nvdash\psi$.
\end{proof}

It is tempting to interpret  $\mathscr{V}(T)$ as a {\em proving measure} of the theory $T$, for the coefficient of $2^{-n}$ in the binary expansion of $\mathscr{V}(T)$ is $1$ if $T\vdash\boldsymbol\psi_n$ and is $0$ otherwise. The following definition and theorem show that one should strongly resist this temptation.

\begin{definition}[$\mathscr{V}^{\langle a,b\rangle}_{\alpha}$]\label{def:wab}
\noindent

\indent
Let $a,b$ be two real numbers such that $b>a\geqslant 0$. Let $\langle\alpha_n\rangle_{n>0}$ be a sequence of positive real numbers such that the series $\sum_{n>0}\alpha_n$ is fast converging with respect to $a/b$, in the sense that for every $n>0$ we have $\sum_{i>n}\alpha_i<\alpha_n(1-{a}/{b})$. One can  take, for example, $\alpha_n=c^{-n}$ for a real number $c$ with $c>1+\frac{b}{b-a}$. Let
$$\sigma_{n}^{\langle a,b\rangle}(T)=\begin{cases}
                      a & \mbox{if } T\nvdash\boldsymbol\psi_n, \\
                      b & \mbox{if } T\vdash\boldsymbol\psi_n;
                    \end{cases}$$
for a theory $T$. Finally, put
$\mathscr{V}^{\langle a,b\rangle}_{\alpha}(T)
=\sum_{n>0}
\alpha_n\sigma_{n}^{\langle a,b\rangle}(T)$.
\hfill \ding{71}
\end{definition}

\begin{theorem}[$\mathscr{V}^{\langle a,b\rangle}_{\alpha}$ satisfies {\sf HP+EP}]\label{thm:onwdash}
\noindent

\indent
If $a,b$, and $\langle\alpha_n\rangle_{n>0}$ are as in Definition~\ref{def:wab} above, then  the mapping $\mathscr{V}^{\langle a,b\rangle}_{\alpha}$ satisfies both {\sf HP} and {\sf EP} for all  theories.
\end{theorem}
\begin{proof}
\noindent

\indent
The analogue of Proposition~\ref{prop:sigma} holds for $\sigma_{n}^{\langle a,b\rangle}$:  for all   theories $T$ and $U$, 

\quad
$T\vdash U   \iff \forall n\!>\!0 \   [\sigma_{n}^{\langle a,b\rangle}(T)\geqslant
 \sigma_{n}^{\langle a,b\rangle}(U)]$.

 So, {\sf HP} holds too: if $T\vdash U$ then $\mathscr{V}^{\langle a,b\rangle}_{\alpha}(T)\geqslant
 \mathscr{V}^{\langle a,b\rangle}_{\alpha}(U)$.

 For showing {\sf EP}, suppose  $\mathscr{V}^{\langle a,b\rangle}_{\alpha}(T)=
 \mathscr{V}^{\langle a,b\rangle}_{\alpha}(U)$.
 By the $\sigma_{n}^{\langle a,b\rangle}$-analogue of Proposition~\ref{prop:sigma}, for proving  $T\equiv U$, it suffices to  show $\sigma_{n}^{\langle a,b\rangle}(T)
= \sigma_{n}^{\langle a,b\rangle}(U)$
 for each $n>0$. If this is not the case, then let $m$ be the minimum $i>0$ such that $\sigma_{i}^{\langle a,b\rangle}(T)
\neq \sigma_{i}^{\langle a,b\rangle}(U)$. Without loss of generality, we can assume that  $\sigma_{m}^{\langle a,b\rangle}(T)=a$ and $\sigma_{m}^{\langle a,b\rangle}(U)=b$. Then

\noindent
\begin{tabular}{rcll}
  $\!\mathscr{V}^{\langle a,b\rangle}_{\alpha}(T)$\!\!\! & $\!=\!$ & $\!\!\!\sum_{0<i<m}
\alpha_i\sigma_{i}^{\langle a,b\rangle}(T)\!+\!\alpha_ma\!+\!\sum_{j>m}
\alpha_j\sigma_{j}^{\langle a,b\rangle}(T)$  &  \!\!\!by Definition~\ref{def:wab}, \\
    & \!$\!=\!$\! & $\!\!\!\sum_{0<i<m}
\alpha_i\sigma_{i}^{\langle a,b\rangle}(U)\!+\!\alpha_ma\!+\!\sum_{j>m}
\alpha_j\sigma_{j}^{\langle a,b\rangle}(T)$   &    \!\!\!by the choice of $m$, \\
   & \!$\!\leqslant\!$\!  & $\!\!\!\sum_{0<i<m}
\alpha_i\sigma_{i}^{\langle a,b\rangle}(U)\!+\!\alpha_ma\!+\!\sum_{j>m}
\alpha_jb$ &   \!\!\!by $\sigma_{j}^{\langle a,b\rangle}(T)\leqslant b$, \\
    &  \!$\!<\!$\! & $\!\!\!\sum_{0<i<m}
\alpha_i\sigma_{i}^{\langle a,b\rangle}(U)\!+\!\alpha_mb$
 &   \!\!\!by the $\alpha_n$'s property,  \\
      &  \!$\!\leqslant\!$\! & $\!\!\!\sum_{0<i<m}
\alpha_i\sigma_{i}^{\langle a,b\rangle}(U)\!+\!\alpha_mb\!+\!\sum_{j>m}
\alpha_j\sigma_{j}^{\langle a,b\rangle}(U)$ &   \!\!\!by $0\leqslant\sigma_{j}^{\langle a,b\rangle}(U)$,  \\
  &  \!$\!=\!$\! & $\!\!\!\mathscr{V}^{\langle a,b\rangle}_{\alpha}(U)$ &   \!\!\!by Definition~\ref{def:wab}.  \\
\end{tabular}

This contradicts the assumption $\mathscr{V}^{\langle a,b\rangle}_{\alpha}(T)=
 \mathscr{V}^{\langle a,b\rangle}_{\alpha}(U)$ above.
\end{proof}

The mapping $\mathscr{V}^{\langle a,b\rangle}_{\alpha}$  can take values in the interval $(0,1)$, for example, when we put  $\alpha_n=c^{-n}$ for some real numbers $a,b,c$ that satisfy  $0\!\leqslant\!a\!\leqslant\!b\!-\!1\!<\!c\!-\!2$ (such as $a\!=\!2,b\!=\!4,c\!=\!7$). But there is no reason to see $\mathscr{V}^{\langle a,b\rangle}_{\alpha}$   as the probability of anything.\footnote{On the contrary, for  $4\!\leqslant\!a\!=\!b\!-\!8$ we have $1\!+\!{b}/{(b-a)}\!=\!2\!+\!{a}/{8}\!<\!a$, so for every $c$ between $2\!+\!{a}/{8}$ and $a$, such as $c\!=\!a\!-\!1$ (for example, $a\!=\!8,b\!=\!16,c\!=\!7$), we will have $\mathscr{V}^{\langle a,b\rangle}_{\alpha}(T)>1$, for every theory $T$ (with $\alpha_n=c^{-n}$ for each $n>0$). Let us also note that for every weight $\mathcal{W}$, the mapping $(1+2^{\mathcal{W}})^{-1}$ ranges over  $(0,1)$ but cannot be interpreted as a probability.}

 \subsubsection{Recursively Enumarable  Theories}\label{sec:re}

 We now consider   {\em recursively enumerable} ({\sc re}) theories; an {\sc re} set is the (possibly infinite) set of the outputs of a fixed  input-free program.
  Let $\mathbb{T}_1, \mathbb{T}_2, \mathbb{T}_3, \cdots$ be an effective list of all the {\sc re} theories (in a fixed language and computing framework).
We  notice  how arbitrary an ({\sf HP+EP})-satisfying weighing of  {\sc re} theories can be:

\begin{remark}[Arbitrary weightings that satisfy {\sf HP} and {\sf EP} for all {\sc re} theories]\label{rem}
\noindent

\indent
{\rm Let us define $\mathscr{U}(\mathbb{T}_n)$ by induction on $n>0$, in such a way that  both  {\sf HP} and {\sf EP} are satisfied. Take $\mathscr{U}(\mathbb{T}_1)$ to be an arbitrary real number. Suppose that $\{\mathscr{U}(\mathbb{T}_1),\cdots,
\mathscr{U}(\mathbb{T}_n)\}$ are defined and that  {\sf HP} and {\sf EP} hold for $\mathbb{T}_1,\cdots,\mathbb{T}_n$. We now define $\mathscr{U}(\mathbb{T}_{n+1})$. If  $\mathbb{T}_{n+1}\equiv \mathbb{T}_{m}$ for some $m\leqslant n$, then let $\mathscr{U}(\mathbb{T}_{n+1})=
\mathscr{U}(\mathbb{T}_{m})$. Now, assume that $\mathbb{T}_{n+1}$ is not equivalent to any of
$\mathbb{T}_1,\cdots,\mathbb{T}_n$. Let $\mathbb{T}_{i_1},\cdots,\mathbb{T}_{i_\Bbbk}$ be all the sub-theories of $\mathbb{T}_{n+1}$ among $\mathbb{T}_1,\cdots,\mathbb{T}_n$ if any; if there is no such sub-theory, then $\Bbbk=0$. Also, let $\mathbb{T}_{j_1},\cdots,\mathbb{T}_{j_\ell}$ be all the super-theories of $\mathbb{T}_{n+1}$ among $\mathbb{T}_1,\cdots,\mathbb{T}_n$ if any; if there is no such super-theory, then $\ell=0$. Since both {\sf HP} and {\sf EP} hold for $\mathbb{T}_1,\cdots,\mathbb{T}_n$ by the induction assumption, then for every $u\leqslant\Bbbk$ and $v\leqslant\ell$ we have $\mathscr{U}(\mathbb{T}_{i_u})<
\mathscr{U}(\mathbb{T}_{j_v})$; note that all $\mathbb{T}_{i_u}$'s are {\em strict} sub-theories of $\mathbb{T}_{n+1}$ and all $\mathbb{T}_{j_v}$'s are {\em strict} super-theories of $\mathbb{T}_{n+1}$. Let $\mathtt{I}\!\!\!\!\mathtt{I} =\max\{\mathscr{U}(\mathbb{T}_{i_1}), \cdots, \mathscr{U}(\mathbb{T}_{i_\Bbbk})\}$ and
$\mathtt{J}\!\!\!\!\mathtt{J}=\min\{\mathscr{U}(\mathbb{T}_{j_1}), \cdots, \mathscr{U}(\mathbb{T}_{j_\ell})\}$; notice that $\max\emptyset=-\infty$ and $\min\emptyset=+\infty$ by convention. Finally, take $\mathscr{U}(\mathbb{T}_{n+1})$ to be any real number between $\mathtt{I}\!\!\!\!\mathtt{I}$ and $\mathtt{J}\!\!\!\!\mathtt{J}$ (note that we have  $\mathtt{I}\!\!\!\!\mathtt{I}<
\mathtt{J}\!\!\!\!\mathtt{J}$ by what was said above;  thus, $\mathtt{I}\!\!\!\!\mathtt{I}<\mathscr{U}
(\mathbb{T}_{n+1})
<\mathtt{J}\!\!\!\!\mathtt{J}$ is quite possible).}
\hfill \ding{71}
\end{remark}

The above construction  can be  computable if our logic is decidable and our theories are all finite. If the underlying logic is not decidable or we do not want to be restricted to finite theories, then the above construction is $\emptyset^{\prime\prime}$-computable (i.e., is computable if one has access to the oracle $\emptyset^{\prime\prime}$). All we have to do  is   check $\mathbb{T}_n\vdash\mathbb{T}_m$, which amounts to the $\Pi_2$-statement $\forall \tau\!\in\!\mathbb{T}_m\exists \pi (\pi\textrm{ is a }\mathbb{T}_n\textrm{-proof of }\tau)$; here we need the effectivity of the list $\langle\mathbb{T}_1, \mathbb{T}_2, \mathbb{T}_3, \cdots\rangle$.  Let us also note that the construction of Remark~\ref{rem} highly  depends 
on how the {\sc re} theories
are ordered; cf.\ also  \cite[p.~577]{Raatikainen98}.

\section{Conclusions}\label{sec:conc}
The {Kolmogorov-Chaitin} complexity is not a good way of weighing theories or sentences since it does not  satisfy {Chaitin}'s heuristic principle  ({\sf HP}), which says that theories can prove lighter sentences only (Definition~\ref{def:hp}).  Neither  does the $\delta$-complexity, which is the {Kolmogorov}-complexity minus the length (Proposition~\ref{prop:nohp}). Due to the  existence of incomparable pairs of theories (none of which derives from the other), the converse of {\sf HP} (that all the  lighter sentences are provable from the theory) cannot hold for real-valued weightings.  {\sf HP} can be satisfied insipidly  by every constant weighing (all the theories and sentences
weigh the same); this constant weighing is trivially computable. To make  {\sf HP} more meaningful, we added the Equivalence Principle ({\sf EP}), which says that only equivalent theories
can have equal weights (Definition~\ref{def:ep}). Note that a consequence of {\sf HP} is that equivalent theories
must weigh equally.  {\sf EP}, equivalently saying that equally weighted theories
are logically equivalent,  is the converse of this statement. If the underlying logic is undecidable, then every weight 
that satisfies {\sf HP} and {\sf EP} should be uncomputable (Theorem~\ref{thm:hpep}). Here, we demonstrated  some  weightings
(Definitions~\ref{def:wdash} and~\ref{def:wab})
that satisfy {\sf HP} and {\sf EP}, and moreover, they are  ({i})~computable if the underlying logic is decidable and the considered theories are all finitely axiomatizable, and  ({ii})~uncomputable if the underlying logic is undecidable. So, regarding the satisfaction of {\sf HP} and {\sf EP}  and the computability of the weighing, this seems to be the best that can be done. Our weightings should not be regarded as the probability of any event, even though they resemble the $\Omega$ number that  is presumed to be the {\em probability of halting}.
All this probably shows that  it takes a genius' idea ({Chaitin}'s constant number) to solve a genius' problem ({Chaitin}'s heuristic principle).

A question that now comes to mind is: Are there any uses for {\sf HP}-satisfying weights of theories? In fact, constructing counter-examples is the most natural way to show unprovability (for example, matrix multiplication shows that the commutativity axiom does not follow from the axioms of group theory). In other words, the most well-known unprovability technique  has been our Definition~\ref{def:mu} (and Theorem~\ref{thm:m}) above. {Lobachevsky}'s geometry is a historic example of a model for the unprovability of the parallel postulate from the rest of the axioms of {Euclid}ean geometry. {G\"odel}'s incompleteness theorems provide monumental examples of the other kinds: the unprovability of {G\"odel}ian sentences (in the first incompleteness  theorem) and the unprovability of the consistency statement (in the second incompleteness theorem), for which {G\"odel} did not construct any model (that satisfies the theory but does not satisfy a {G\"odel}ian sentence or the consistency statement); his proof was totally syntactic. So far, we do not know if there ever was an unprovability result in the history of mathematics that used     an ({\sf HP+EP})-satisfying weighing of   theories.

\part{Halting Probability:  To Compute Or Not To Compute}
\author{Manizheh~Jalilvand, Behzad~Nikzad}
  \shorttitle{\sl On Chaitin's Halting Probability}
\section{Motivation}\label{sec:mp}
A coin is said to be {\em fair}, when the probability of getting a head ({\tt H}) by tossing it is equal to the probability of getting a tail ({\tt T}). Thus, each probability is $\frac{1}{2}$ since they should add up to $1$ by Kolmogorov's axioms of probability measure. So, the probability of getting a first tail by tossing a fair coin two times (that is, getting either {\tt TH} or {\tt TT}) is again $\frac{1}{2}$. Let us toss a fair coin one or two times and ask what the probability of getting either one head (that is, {\tt H}) or two sides that begin with a tail (that is, {\tt TH} or {\tt TT}) is. To choose the number of tosses, we provide an urn containing two similar balls, each with a unique label of {\tt 1} or {\tt 2}. We close our eyes and pick up a ball from the urn. If the ball is labeled {\tt 1}, then we toss the fair coin once; if it is labeled {\tt 2}, then we toss twice. What is the probability that we get {\tt H}, {\tt TH}, or {\tt TT}?

Whatever that is, let us consider the complement of this event. What is the probability of getting {\tt T}, {\tt HT}, or {\tt HH} by tossing a fair coin once or twice randomly? The probability of {\tt H} should be equal to the probability of {\tt T}; call their common value $q$. This is the definition of a fair coin. So should be the probabilities of {\tt TH}, {\tt TT}, {\tt HT}, and {\tt HH}; call their common value $r$. So, the probability of $E=\{\texttt{H},\texttt{TH},\texttt{TT}\}$ is $q+2r$, and so is the probability of its complement $E^\complement=\{\texttt{T},\texttt{HT},\texttt{HH}\}$. Thus, the probability of $E$ and that of $E^\complement$ should be $\frac{1}{2}$. This holds even if the probability of getting the ball with label {\tt 1} out of the urn is not equal to the probability of getting the other ball with label {\tt 2}. That is to say that the probability of $E$ is $\frac{1}{2}$, no matter the values of $q$ and $r$. Notice that the sample space is $\{\texttt{H},\texttt{T},\texttt{HH},\texttt{HT},\texttt{TH},\texttt{TT}\}$, so we should have $2q+4r=1$, thus the probability of $E$ is $\mathfrak{p}(E)=q+2r=\frac{2q+4r}{2}=\frac{1}{2}$. All we assumed here was that (1)~$\mathfrak{p}(\texttt{H})=\mathfrak{p}(\texttt{T})=q$, and (2)~$\mathfrak{p}(\texttt{HH})=\mathfrak{p}(\texttt{HT})=
\mathfrak{p}(\texttt{TH})=\mathfrak{p}(\texttt{TT})=r$.

Therefore, the probability of our event $E$ is not equal to $1$, but its Omega is so: $\varOmega_E=\frac{1}{2}+\frac{1}{4}+\frac{1}{4}=1$. In the literature, the Omega of a set $S$ of finite sequences of {\tt 0}'s and {\tt 1}'s, so-called {\em binary strings}, is defined as $\varOmega_S=\sum_{\sigma\in S}2^{-|\sigma|}$, where $|\sigma|$ denotes the length of $\sigma$ ($\varOmega_S$ is also called the {\em weight} of the set $S$; see, e.g., \cite[p.~201]{DH10}). Let us note that we have identified {\tt H} with {\tt 0} and {\tt T} with {\tt 1}.

 So,  let us write  $\tilde{E}=\{\texttt{0},\texttt{10},\texttt{11}\}$  (for which  we have $\varOmega_{\tilde{E}}=1$) and now toss our fair coin for a randomly finite number of times, and
compute the probability of getting a string from $\tilde{E}$. Our sample space is now $\{\texttt{0},\texttt{1}\}^+$, the set of all nonempty binary strings. Let $\pi_1$ be the common probability of {\tt 0} and {\tt 1} (recall the definition of a fair coin). Let $\pi_2$ be the common probabilities of {\tt 00}, {\tt 01}, {\tt 10}, and {\tt 11}. For each $n>0$, let $\pi_n$ be the probability of the binary strings with length $n$. Since there are $2^n$  such strings, then we should have $\sum_{n=1}^{\infty}2^n\pi_n=1.$
Now, the probability of $\tilde{E}$ is $\pi_1+2\pi_2$, which is much less than $1$, since
$\pi(\tilde{E})=\pi_1+2\pi_2
=(1/2)({2\pi_1+4\pi_2})\leqslant
(1/2)({\sum_{n=1}^{\infty}2^n\pi_n})=
(1/2)<1=\varOmega_{\tilde{E}}.$

\section{Preliminaries}
\subsection{Some Computational Preliminaries}
\begin{definition}[binary code,  length, {\tt ASCII} code]\label{def:sigma}
\noindent

\indent
Let $\boldsymbol\Sigma=\{\texttt{0},\texttt{1}\}^{\boldsymbol+}$ be the set of all the (nonempty) finite strings of the symbols (binary bits) $\texttt{0}$ and $\texttt{1}$.
For a string $\sigma\in\boldsymbol\Sigma$, let $|\sigma|$ denote its length.
Every program has a unique {\tt ASCII} code\footnote{{\tt ASCII}: {\em American Standard Code for Information Interchange}, The Extended 8-bit Table Based on   Windows-1252  ({\sf 1986}), available at      \url{https://www.ascii-code.com/}}   which is a binary string. This is called the binary code of the program.
\hfill\ding{71}\end{definition}

\begin{example}[binary code,  length, {\tt ASCII} code]
\noindent

\indent
The object $\texttt{01001}$ is a binary string, and its length is  $|\texttt{01001}|\!=\!5$.
The {\tt ASCII} code of the symbol @ is {\tt  01000000},  and {\tt 00100000} is the {\tt ASCII} code of the {\em blank space}, produced by the {\tt space} {\tt bar} on  the keyboard.
\hfill\ding{71}\end{example}

\begin{remark}[the empty string]
\noindent

\indent
We exclude the empty string with length $0$, which is usually included in   theory of formal languages. So, strings have all positive lengths.
\hfill\ding{71}
\end{remark}

\begin{example}[binary code of a program]
\noindent

\indent
Let us consider   the   command

\qquad {\tt BEEP}

\noindent
in, e.g., the {\tt BASIC} programming language; it produces the actual ``beep'' sound through the sound card of the computer hardware. The binary code of this command is the concatenation of the {\tt ASCII} codes of the capital letters {\tt B} (which is \texttt{01000010}), {\tt E} (which is \texttt{01000101}), {\tt E} (the same), and {\tt P} (which is \texttt{01010000}), building together  the following finite binary string:

\qquad $\texttt{01000010010001010100010101010000}$.
\hfill\ding{71}\end{example}

Next, we consider input-free programs and their halting problem, since for defining $\boldsymbol\Omega$, Chaitin gave the main idea as,  ``The idea is you generate each bit of a program by tossing a coin and
ask what is the probability that it halts.''  \cite[p.~151]{Chaitin95}.
By ``program,'' Chaitin meant an {\em input-free} program,  and by  ``bit'' of a program,  he meant   any of  the \texttt{0}'s and \texttt{1}'s in its binary ({\tt ASCII}) code.

\begin{example}[programs: input-free, halting, and non-halting]\label{ex:progs}
\noindent

\indent
Consider the following three programs over a fixed programming language, where the variables $i$ and $n$ range over the natural numbers.

\begin{center}
\begin{tabular}{||l||l||l||}
 \hline
  \textup{Program 1}
  & \textup{Program 2} & \textup{Program 3}
   \\
     \hline \hline
  \textup{\texttt{BEGIN}}
  &
  \textup{\texttt{BEGIN}} &
  \textup{\texttt{BEGIN}}
  \\
  \;\;\textup{\texttt{LET}} $n\!:=\!1$
  &
  \;\;\textup{\texttt{LET}} $n\!:=\!1$ & \;\;\textup{\texttt{INPUT}} $i$
  \\
  \;\;\;\;\textup{\texttt{WHILE}} $n\!>\!0$ \textup{\texttt{DO}}
  &
  \;\;\;\;\textup{\texttt{WHILE}} $n\!<\!9$ \textup{\texttt{DO}} &
  \;\;\;\;\textup{\texttt{WHILE}} $i\!<\!9$ \textup{\texttt{DO}}
   \\
   \;\;\;\;\;\;\textup{\texttt{begin}}
  &
  \;\;\;\;\;\;\textup{\texttt{begin}}  &
  \;\;\;\;\;\;\textup{\texttt{begin}}
  \\
  \;\;\;\;\;\;\;\; \textup{\texttt{PRINT}} $n$
  &
  \;\;\;\;\;\;\;\; \textup{\texttt{PRINT}} $n$  &
  \;\;\;\;\;\;\;\;\textup{\texttt{PRINT}} $i$
   \\
  \;\;\;\;\;\;\;\; \textup{\texttt{LET}} $n\!:=\!n\!+\!1$
  &
  \;\;\;\;\;\;\;\; \textup{\texttt{LET}} $n\!:=\!n\!+\!1$  &
  \;\;\;\;\;\;\;\;\textup{\texttt{LET}} $n\!:=\!i\!+\!1$ \  \\
   \;\;\;\;\;\;\textup{\texttt{end}}
  &
  \;\;\;\;\;\;\textup{\texttt{end}}  &
  \;\;\;\;\;\;\textup{\texttt{end}} \\
  \textsc{\texttt{END}}
  &
  \textup{\texttt{END}} &
  \textup{\texttt{END}} \\
  \hline
\end{tabular}
\end{center}

Program~1 and Program~2 do not take any input, while Program~3 takes some input (from the user) and then starts running.
Program~1 never halts (loops forever) after it starts running, but Program~2 halts eventually (when $n$ reaches $9$).
Program~3 takes the input $i$ and halts
on some values of $i$ (when $i\!\geqslant\!9$) and loops forever on others (when $i\!<\!9$).
\hfill\ding{71}\end{example}

As Chaitin took    ``programs''  for  ``input-free programs,'' we will also use these terms interchangeably; so, we disregard the programs that take some inputs and consider only input-free programs. The number
  $\boldsymbol\Omega$ was defined   by Chaitin as follows (\cite[p.~150]{Chaitin95}):
\begin{itemize}
\item[]\begin{quote}
{
What exactly is the halting probability? I've written down an expression for it: $\Omega\!=\!\sum_{p\textrm{ halts}}2^{-|p|}$. [\!...\!] If you generate a computer program at random by tossing a coin for each bit of the program, what is the chance that the program will halt? You're thinking of programs as bit strings, and you generate each bit by an independent toss of a fair coin}.
\end{quote}
\end{itemize}

Actually, for an arbitrary set $S$ of binary strings, one can define $\varOmega_S$ as follows:

\begin{definition}[$\varOmega_S$]\label{def:omega-s}
\noindent

\indent
For a set of binary strings   $S\subseteq\boldsymbol\Sigma$, let $\varOmega_S\!=\!\sum_{\sigma\in S}2^{-|\sigma|}$.
\hfill\ding{71}\end{definition}

\begin{example}[${\varOmega}_S$]\label{ex:omega-s}
\noindent

\indent
We have $\varOmega_{\{\texttt{0}\}}\!=\!\frac{1}{2}$,  $\varOmega_{\{\texttt{0},\texttt{00}\}}\!=\!\frac{3}{4}
\!=\!\varOmega_{\{\texttt{1},\texttt{00}\}}$, and
 $\varOmega_{\{\texttt{0},\texttt{1},\texttt{00}\}}
 \!=\!\frac{5}{4}$.

 We also have
 $\varOmega_{\mathcal{C}}\!=\!1$, where
 $\mathcal{C}=\{\texttt{1},\texttt{00},\texttt{010},\texttt{0110},
\texttt{01110},\texttt{011110},\dots\}$.
\hfill\ding{71}\end{example}

\begin{definition}[$\mathbb{P},\mathbb{H}$]\label{def:hp}
\noindent

\indent
Let $\mathbb{P}$ denote the set of the binary codes of all the input-free programs over a fixed programming language.
Over that fixed language, let $\mathbb{H}$ denote the set of the binary codes of all those input-free programs that halt after running (eventually stop; do not loop forever).
\hfill\ding{71}\end{definition}

\begin{definition}[the Omega number]\label{def:omega-n}
\noindent

\indent
Let  $\boldsymbol\Omega$ be the number $\varOmega_{\mathbb{H}}$ (see Definitions~\ref{def:omega-s} and~\ref{def:hp}).
 \hfill\ding{71}\end{definition}

This finishes our mathematical definition of the Omega Number. Let us notice that ``the precise numerical value of [$\boldsymbol\Omega$] depends on the choice [of the fixed] programming language'' \cite[p.~236]{CalCha10}.

\subsection{Some Mathematical Preliminaries}

\,

The number $\boldsymbol\Omega$ as 
in Definition~\ref{def:omega-n} may not lie in 
 $[0,1]$, and so it  may not be the probability of anything. Chaitin warned this in \cite[p.~150]{Chaitin95}:
\begin{itemize}
\item[]\begin{quote}
{
there's a technical detail which is very important and didn't work in the early version of algorithmic information theory. You couldn't write this: $\Omega\!=\!\sum_{p\textrm{ halts}}2^{-|p|}$. It would give infinity. The technical detail is that no extension of a valid program is a valid program. Then this sum $\Omega\!=\!\sum_{p\textrm{ halts}}2^{-|p|}$ turns out to be between zero and one. Otherwise it turns out to be infinity. It only took ten years until I got it right. The original 1960s version of algorithmic information theory is wrong. One of the reasons it's wrong is that you can't even define this number. In 1974 I redid algorithmic information theory with `self-delimiting' programs and then I discovered the halting probability, $\Omega$.}
\end{quote}
\end{itemize}

\begin{definition}[prefix-free]\label{def:prefi}
\noindent

\indent
A set of binary strings is prefix-free when no element of it is a proper prefix of another element of it.
\hfill
\ding{71}\end{definition}

\begin{example}[prefix-free]\label{ex:prefixf}
\noindent

\indent
The sets  $Z=\{\texttt{0}\}$ and $F=\{\texttt{1},\texttt{00}\}$ are both prefix-free, but their union
 $Z\cup F=\{\texttt{0},\texttt{1},\texttt{00}\}$ is not, since \texttt{0} is a prefix of
 \texttt{00}; neither is the set $\{\texttt{0},\texttt{00}\}$. The    set $\mathcal{C}=\{\texttt{1},\texttt{00},\texttt{010},\texttt{0110},
\texttt{01110},\texttt{011110},\dots\}$ is also prefix-free  (see Example~\ref{ex:omega-s}).
\hfill\ding{71}\end{example}

The Omega of every prefix-free set is 
non-greater than one. This is known as Kraft's Inequality \cite{Kraft49} and will be proved in the following (Proposition \ref{prop:kraft}).

\begin{definition}[binary expansion in base 2]\label{def:base2}
\noindent

\indent
Every natural number has a binary expansion (in base 2), which is a finite binary string that starts with $\texttt{1}$; that is to say that every $n\!\in\!\mathbb{N}$ can be written as $n\!=\!(\!\!(x_kx_{k-1}\dots x_2x_1x_0)\!\!)_{\bf 2}\!=\!
\sum_{i=0}^{k}x_i2^i$, where $x_i\!\in\!\{\texttt{0},\texttt{1}\}$, for $i\!=\!0,1,2,\dots,k\!-\!1$, and $x_k\!=\!\texttt{1}$.
Every real number $\alpha$ in the unit interval $(0,1]$ has a binary expansion (in base 2) as $\alpha\!=\!(\!\!(\texttt{0}\centerdot x_1x_2x_3\dots)\!\!)_{\bf 2}\!=\!\sum_{i=1}^{\infty}x_i2^{-i}$, where $x_i\!\in\!\{\texttt{0},\texttt{1}\}$, for $i\!=\!1,2,3,\dots$ (see \cite{Linde24}). This expansion could be finite or infinite.
\hfill\ding{71}\end{definition}

\begin{example}[binary expansion in base 2]\label{ex:base2}
\noindent

\indent
We have
$9\!=\!(\!\!(\texttt{1001})\!\!)_{\bf 2}$,
$26\!=\!(\!\!(\texttt{11010})\!\!)_{\bf 2}$,
$41\!=\!(\!\!(\texttt{101001})\!\!)_{\bf 2}$, and clearly we have
$1\!=\!(\!\!(\texttt{0}\centerdot \texttt{111}\dots)\!\!)_{\bf 2}$. Also
$\frac{9}{32}\!=\!(\!\!(\texttt{0}\centerdot \texttt{01001})\!\!)_{\bf 2}\!=\!(\!\!(\texttt{0}\centerdot \texttt{01000111}\dots)\!\!)_{\bf 2}$.
\hfill\ding{71}\end{example}

\begin{remark}[uniqueness]\label{rem:base2}
\noindent

\indent
Every natural number has a unique binary expansion,  which is a finite binary string. The infinite binary expansion of any real number in $(0,1]$ is unique.
\hfill\ding{71}\end{remark}


\begin{definition}[$\mathbb{I}_\sigma,\mathfrak{L}$]\label{def:iell}
\noindent

\indent
For a binary string $\sigma\!\in\!\boldsymbol\Sigma$, let $\mathbb{I}_\sigma$ be the interval $\big((\!\!(\texttt{0}\centerdot\sigma)\!\!)_{\bf 2} \boldsymbol,(\!\!(\texttt{0}\centerdot\sigma\texttt{111}\!
\dots)\!\!)_{\bf 2}
\big]$, which consists of all the real numbers in $(0,1]$  whose infinite binary expansions after $\texttt{0}\centerdot$ contain $\sigma$ as a prefix (cf. \cite{Linde24}).
Denote the Lebesgue measure on 
$\mathbb{R}$ by $\mathfrak{L}$.
\hfill\ding{71}\end{definition}

\begin{example}[$\mathbb{I}_\sigma,\mathfrak{L}$]\label{ex:iell}
\noindent

\indent
We have $\mathbb{I}_{\{\texttt{0}\}}\!=\!(0,\frac{1}{2}]$, $\mathbb{I}_{\{\texttt{1}\}}\!=\!(\frac{1}{2},1]$, $\mathbb{I}_{\{\texttt{00}\}}\!=\!(0,\frac{1}{4}]$, and
$\mathbb{I}_{\{\texttt{01001}\}}\!=\!(\frac{9}{32},\frac{5}{16}]$. The Lebesgue measures (lengths) of these intervals are $\mathfrak{L}(\mathbb{I}_{\{\texttt{0}\}})\!=\!\frac{1}{2}$,
$\mathfrak{L}(\mathbb{I}_{\{\texttt{1}\}})\!=\!\frac{1}{2}$,
$\mathfrak{L}(\mathbb{I}_{\{\texttt{00}\}})\!=\!\frac{1}{4}$, and finally
$\mathfrak{L}(\mathbb{I}_{\{01001\}})\!=\!\frac{1}{32}$.
\hfill\ding{71}\end{example}

\begin{lemma}[$\mathbb{I}_\sigma\!\!\setminus\!\!\{1\}
\!\!\subseteq\!\!(0,1)$, $\mathfrak{L}(\mathbb{I}_\sigma)\!\!=\!\!2^{-|\sigma|}$, $\mathbb{I}_\sigma\!\cap\!\mathbb{I}_{\sigma'}$,   $\mathfrak{L}(\bigcup_{\sigma\in S}\mathbb{I}_\sigma)\!\!=\!\!\varOmega_S$]
\label{lem:big}
\noindent

\indent
Let $\sigma,\sigma'\!\in\!\boldsymbol\Sigma$ be fixed.

(1) The interval $\mathbb{I}_\sigma$ is a half-open subinterval of $(0,1]$, i.e., $\mathbb{I}_\sigma\!\subseteq\!(0,1]$.

(2) The length of $\mathbb{I}_\sigma$ is $\frac{1}{2^{|\sigma|}}$, i.e., $\mathfrak{L}(\mathbb{I}_{\sigma})\!=\!2^{-|\sigma|}$.

(3) If $\sigma$ is not a prefix of $\sigma'$ and $\sigma'$ is not a prefix of $\sigma$, then   $\mathbb{I}_\sigma\cap \mathbb{I}_{\sigma'}\!=\!\emptyset$.

(4) If $S\subseteq\boldsymbol\Sigma$ is prefix-free, then  $\mathfrak{L}(\bigcup_{\varsigma\in S}\mathbb{I}_\varsigma)\!=\!\varOmega_S$.
\end{lemma}
\begin{proof}
\noindent

\indent
(1) is trivial; for (2) notice that

\begin{tabular}{rcl}
  $\mathfrak{L}(\mathbb{I}_{\sigma})$ & $=$ & $(\!\!(\texttt{0}\centerdot\sigma\texttt{111}\!
\dots)\!\!)_{\bf 2}-
 (\!\!(\texttt{0}\centerdot\sigma)\!\!)_{\bf 2}$     \\
    & $=$ & $(\!\!(\texttt{0}\centerdot    \underbrace{\texttt{0}\dots
    \texttt{0}}_{|\sigma|\textrm{-times}}
    \texttt{111}\!\dots)\!\!)_{\bf 2}$       \\
    & & \\[-1em]
   & $= $  & $\sum_{j=1}^{\infty}2^{-(|\sigma|+j)}$   \\
      & $= $  & $2^{-|\sigma|}$.
\end{tabular}

(3) If $\alpha\!\in\!\mathbb{I}_\sigma\cap \mathbb{I}_{\sigma'}$, then  $\alpha\!=\!(\!\!(\texttt{0}\centerdot\sigma x_1x_2x_3\!\dots)\!\!)_{\bf 2}$ and $\alpha\!=\!(\!\!(\texttt{0}\centerdot{\sigma'} y_1y_2y_3\!\dots)\!\!)_{\bf 2}$, where the sequences $\{x_i\}_{i>0}$ and $\{y_i\}_{i>0}$ are not all 0. Thus, by Remark~\ref{rem:base2}, the identity
$(\!\!(\texttt{0}\centerdot\sigma x_1x_2x_3\!\dots)\!\!)_{\bf 2}\!=\!(\!\!(\texttt{0}\centerdot{\sigma'} y_1y_2y_3\!\dots)\!\!)_{\bf 2}$ implies that either $\sigma$ should be a prefix of ${\sigma'}$ or ${\sigma'}$ should be a prefix of $\sigma$.

(4) We have $\mathfrak{L}(\bigcup_{\varsigma\in S}\mathbb{I}_\varsigma)\!=\!\sum_{\varsigma\in S}\mathfrak{L}(\mathbb{I}_\varsigma)$ since $\mathbb{I}_\varsigma$'s are pairwise disjoint by item (3). The result follows now from item (2) and Definition~\ref{def:omega-s}.
\end{proof}

\begin{proposition}[Kraft's inequality, 1949]\label{prop:kraft}
\noindent

\indent
For every prefix-free   $S\subseteq\boldsymbol\Sigma$, we have ${\varOmega}_S\!\leqslant\!1$.
\end{proposition}
\begin{proof}
\noindent

\indent
By Lemma~\ref{lem:big}, item (4), we have ${\varOmega}_S\!=\!\mathfrak{L}(\bigcup_{\sigma\in S}\mathbb{I}_\sigma)$, and
$\bigcup_{\sigma\in S}\mathbb{I}_\sigma\!\subseteq(0,1]$ holds
by item (1) of Lemma~\ref{lem:big}. Therefore, ${\varOmega}_S\!\leqslant\!\mathfrak{L}(0,1]\!=\!1$.
\end{proof}

For an alternative proof of Proposition~\ref{prop:kraft}, see, e.g., \cite[Thm.~11.4, pp.~182-3]{RS07}.
Let us notice that the converse of Kraft's inequality is not true, since, as we saw in Examples~\ref{ex:omega-s} and~\ref{ex:prefixf},
${\varOmega}_{\{\texttt{0},\texttt{00}\}}
\!=\!\frac{3}{4}\!<\!1$, but the set $\{\texttt{0},\texttt{00}\}$ is not prefix-free.

One way to ensure that the set of all  the programs becomes prefix-free is to adopt   the following convention:

\begin{convention}[prefix-free programs]\label{conv}
\noindent

\indent
Every program ends with the ``{\tt END}'' command \textup{(see \cite[p.~3]{Monro78})}. This command can appear nowhere else in the program, only at the very end.
\hfill\ding{71}
\end{convention}

Every other sub-routine may start with ``{\tt begin}'' and finish  with ``{\tt end},'' just like   the programs of  Example~\ref{ex:progs}.

\begin{example}[prefix-free programs]\label{ex:conv}
\noindent

\indent
Program~$i$ in the following table is a prefix of Program~$i\!i$ (and a suffix of Program~$i\!i\!i$).

\begin{center}
\begin{tabular}{||l||l||l||}
 \hline
  \textup{Program $i$}
  & \textup{Program $i\!i$} &   \textup{Program $i\!i\!i$}
   \\
     \hline \hline
  \textup{\texttt{BEEP}}
  &
  \textup{\texttt{BEEP}} & \textup{\texttt{PRINT} ``{\sf error}!"}
  \\
  &
  \textup{\texttt{PRINT} ``{\sf error}!"} & \textup{\texttt{BEEP}}
  \\
  \hline
\end{tabular}
\end{center}

With Convention~\ref{conv}, the programs should look like the following:

\begin{center}
\begin{tabular}{||l||l||l||}
 \hline
  \textup{Program I}
  & \textup{Program I\!I} & \textup{Program I\!I\!I}
   \\
     \hline \hline
  \textup{\texttt{BEEP}}
  &
  \textup{\texttt{BEEP}} & \textup{\texttt{PRINT} ``{\sf error}!"}
  \\
   \textup{\texttt{END}} &
  \textup{\texttt{PRINT} ``{\sf error}!"} & \textup{\texttt{BEEP}}
  \\
  &
  \textup{\texttt{END}} & \textup{\texttt{END}}
  \\
  \hline
\end{tabular}
\end{center}

\noindent
Program I is not a prefix of Program I\!I (though, even with the above convention, Program I is a suffix of Program I\!I\!I, which is not a problem).
\hfill\ding{71}\end{example}

From now on, let us be given a \textbf{\textsl{fixed programming language}} by Convention~\ref{conv}.

\section{To Be A Probability, Or Not To Be}

\begin{question}[is ${\varOmega}_S$ a probability?]\label{q:big}
\noindent

\indent
Why can  ${\varOmega}_S$ be interpreted as the {\em probability} that a randomly given binary string $\sigma\!\in\!\boldsymbol\Sigma$ belongs to $S$? Even when  $S\!\subseteq\!\boldsymbol\Sigma$ is a prefix-free set.
\hfill\ding{71}\end{question}

Let us repeat that the number ${\varOmega}_S$ could be greater than one for some sets $S$ of finite binary strings (Example~\ref{ex:omega-s}), but if the set $S$ is prefix-free, then ${\varOmega}_S$ is a number between $0$ and $1$ (Proposition~\ref{prop:kraft}).  Let us also note that ${\varOmega}$ satisfies Kolmogorov's axioms of a measure: ${\varOmega}_{\bigcup_i\!S_i}\!=\!\sum_i{\varOmega}_{S_i}$ for every family $\{S_i\}_i$ of  pairwise disjoint sets; thus, ${\varOmega}_\emptyset\!=\!0$.
But  it is not a probability measure.
 Restricting the sets to the prefix-free ones
will not solve the problem, as they are not closed under unions (Example~\ref{ex:prefixf}).

Now that, by Convention~\ref{conv}, all the programs are prefix-free, a special case of
Question~\ref{q:big} is:

\begin{question}[is $\boldsymbol\Omega$ a halting probability?]\label{q:omega}
\noindent

\indent
Why can 
$\boldsymbol\Omega$ be said to be the {\em halting probability} of the randomly chosen finite binary strings?
\hfill\ding{71}\end{question}

Unfortunately, many scholars seem to have believed that the number $\boldsymbol\Omega$ is the halting probability of input-free programs; see,  e.g.,  \cite{Gardner79,BFGM06,RS07,CalCha10,Schmidhuber22}.
Even though the ${\varOmega}$'s of prefix-free sets are non-greater than one, ${\varOmega}$ is not a probability measure, even when restricted to the prefix-free sets, as those sets are not closed under disjoint unions.
Restricting the sets to   the subsets of a fixed prefix-free set whose ${\varOmega}$ is 1 (such as $\mathcal{C}$ in Example~\ref{ex:omega-s})
can solve the problem. But for the input-free programs, even with Convention~\ref{conv}, we do not have this possibility:

\begin{lemma}[${\varOmega}_{\mathbb{P}}\!\neq\!1$]
\label{lem:pi}
\noindent

\indent
${\varOmega}_{\mathbb{P}}\!<\!1$.
\end{lemma}
\begin{proof}
\noindent

\indent
Find a letter or a short string of letters  (such as {\tt X} or {\tt XY}, etc.) that is not a prefix of any command, and no program can be a  prefix of it.
Let $\mathfrak{X}$ be its {\tt ASCII} code, and put
$P'\!=\!\mathbb{P}\cup\{\mathfrak{X}\}$. The set $P'$ is still prefix-free, and so Kraft's inequality (Proposition~\ref{prop:kraft}) can be applied to it: ${\varOmega}_{\mathbb{P}}\!+\!
2^{-|\mathfrak{X}|}
\!=\!{\varOmega}_{P'}
\!\leqslant\!1$.
Since $2^{-|\mathfrak{X}|}
\!>\!0$, then we have ${\varOmega}_{\mathbb{P}}\!<\!1$.
\end{proof}

For making $\varOmega$ a probability measure, we suggest a two-fold idea:

(1) We consider sets of input-free programs only,  and

(2) We divide their Omega by ${\varOmega}_{\mathbb{P}}$ to get a  probability measure.

\begin{definition}[$\mho_S$]\label{def:mho}
\noindent

\indent
For a set $S\subseteq\mathbb{P}$ of input-free programs, let
$\mho_S\!=\!\dfrac{{\varOmega}_S}{{\varOmega}_{\mathbb{P}}}.$
\hfill\ding{71}\end{definition}

It is easy to verify that this {\em is} a probability measure: we have $\mho_{\emptyset}\!=\!0$,
$\mho_{\mathbb{P}}\!=\!1$, and for every indexed family $\{S_i\!\subseteq\!\mathbb{P}\}_i$ of  pairwise disjoint sets of input-free programs,   we have
$\mho_{\bigcup_i\!S_i}\!=\!\sum_i\mho_{S_i}$.

\subsection{A Recapitulation}
Let us recapitulate. The number $\boldsymbol\Omega$ (Definition~\ref{def:omega-n}) was meant to be ``the probability that a computer program whose
bits are generated one by one by independent tosses of a fair coin will eventually halt'' \cite[p.~236]{CalCha10}.
But the fact of the matter is that if we generate a finite binary code by tossing a fair coin bit by bit,
then it is very probable that the resulted string is not the binary code of  a program at all. It is also
highly probable that it is the code of a program that takes some inputs (see Example~\ref{ex:progs}). Lastly, if the
generated finite binary string is the binary code of an input-free program, then we are allowed to ask whether it will eventually halt after running. After all this
contemplation, we may start defining or calculating the probability of halting.

The way $\boldsymbol\Omega$ was defined
works for any prefix-free set of finite binary strings (Definition~\ref{def:omega-s}). Kraft’s inequality (Proposition~\ref{prop:kraft}) ensures that the
number ${\varOmega}_S$, for every prefix-free set $S$, lies in the interval $[0, 1]$. But why on earth can ${\varOmega}_S$ be called
the probability that a randomly given finite binary string belongs to $S$? (Question~\ref{q:big}). The class of all
prefix-free sets is not closed under disjoint unions  (Example~\ref{ex:prefixf}), and
  there is no {\em sample space} for the proposed measure: the ${\varOmega}$ of all the binary codes of the input-free programs is not equal to $1$ (Lemma~\ref{lem:pi}), even though that set  is prefix-free by Convention~\ref{conv}. Summing up, there is no measure to see that  $\boldsymbol\Omega$ is   the halting probability of a randomly given finite binary string, and the answer to Question~\ref{q:omega} is a big ``no''.

Even though ${\varOmega}$ satisfies Kolmogorov's axioms of a measure, it is not a probability measure, as some sets get measures bigger than one.
Restricting the sets to the prefix-free ones will not solve the problem, as they are not closed under union. Restricting the sets to   the subsets of a fixed prefix-free set whose ${\varOmega}$ is $1$
can solve the problem by making ${\varOmega}$ a probability measure. So can   restricting the sets to   the subsets of a fixed prefix-free set (such as $\mathbb{P}$)  and then dividing the ${\varOmega}$'s of its subsets by the ${\varOmega}$ of that fixed set (just like Definition~\ref{def:mho}).

This was our proposed remedy. Take the sample space to  be $\mathbb{P}$, the set of the binary codes of all the input-free programs. Then, for every set $S$ of (input-free) programs
 ($S\subseteq\mathbb{P}$),
let
$\mho_S\!=\!
\frac{\varOmega_S\!}{\!\varOmega_{\mathbb{P}}}$ (Definition~\ref{def:mho}).
This is a {\em real} probability measure that satisfies Kolmogorov's axioms. Now, the {\em new} halting probability is  $\boldsymbol\mho\!=\!\mho_{\mathbb{H}}\!=\!
\frac{\boldsymbol\Omega\!}{\!\varOmega_{\mathbb{P}}}$.
Dividing $\boldsymbol\Omega$ by a computable real number ($\varOmega_{\mathbb{P}}$) does make it look more like a (conditional) probability, but will not cause it to lose any of the non-computability or randomness properties. Our upside-down Omega, $\boldsymbol\mho$,  should have most (if not all) of the properties of $\boldsymbol\Omega$ established in the literature.

\section{Un/Computing the Halting Probability}

Let us see read  through one of Chaitin's books (\cite[p.~112, original emphasis]{Chaitin05}):
\begin{itemize}
\item[]\begin{quote}
{
let's put {\bf all possible} programs in a bag, shake it up, close our eyes, and pick out a program. What's the probability that this program that we've just chosen at random will eventually halt? Let's express that probability as an infinite precision binary real between zero and one. [\!...\!] You sum for each program that halts the probability of getting precisely that program by chance:
$\Omega\!=\!\sum_{\textrm{\tiny program }p \ \textrm{\tiny   halts}}2^{-(\textrm{\tiny size in bits of }p)}$.
Each $k$-bit self-delimiting program $p$ that halts
contributes $1/{2^k}$ to the value of $\Omega$. The self-delimiting program proviso is crucial: Otherwise the halting
probability has to be defined for programs of {\bf each particular size}, but
it cannot be defined over {\bf all} programs of {\bf arbitrary size}.}
\end{quote}
\end{itemize}

We are in partial agreement with Chaitin on the following matter:

\begin{lemma}[halting probability of i.f. prog.'s with a fixed length]
\label{lem:agree}
\noindent

\indent
The halting probability of all the input-free programs with a fixed length $\ell$ is equal to $\sum^{|p|=\ell}_{p\textrm{ halts}}2^{-|p|}.$
\end{lemma}
\begin{proof}
\noindent

\indent
Fix a number $\ell$. The probability of getting a fixed binary string of length $\ell$ by tossing a fair coin (whose one side is `{\tt 0}' and the other `{\tt 1}') is $\frac{1}{2^{\ell}}$, and the halting probability of the input-free programs with  length $\ell$ is $$\dfrac{\textrm{the number of halting programs of length }\ell}{\textrm{the number of all binary strings of length }\ell}\!=\!\dfrac{\#\{p\!\in\!\mathbb{P}\!: p\textrm{ halts } \&\  |p|\!=\!\ell\}}{2^\ell}, $$
since there are $2^\ell$ binary strings of length $\ell$ (see~\cite{Linde24}). Thus, the halting probability of programs with length $\ell$ can be written as  $\sum^{|p|=\ell}_{p\textrm{ halts}}2^{-|p|}$.
\end{proof}

Our disagreement is about the halting probability of input-free programs, not with a fixed length but with an arbitrary length.

\begin{definition}[$\mathcal{N}(\ell)$]\label{def:Nl}
\noindent

\indent
Let $\mathcal{N}(\ell)$ be the number of halting input-free programs of length $\ell$.
\hfill\ding{71}\end{definition}

So, the number $\boldsymbol\Omega$ can be written as $\sum_{\ell=1}^\infty\mathcal{N}(\ell)2^{-\ell}$; see \cite[p.~1]{Schmidhuber22}. By what we quoted above, from \cite{Chaitin05}, according to Chaitin (and almost everybody else), the halting probability of programs   is $\sum_{\ell=1}^{\infty}\mathcal{N}(\ell)2^{-\ell}
\!=\!\sum_{p\textrm{ halts}}2^{-|p|}(\!=\!\boldsymbol\Omega)$!
We believe this to be an error, since we can show that $\boldsymbol\Omega$ is not the halting probability {\em under any measure}:

\begin{theorem}[halting probability \textrm{[by any measure]} $\!<\!\boldsymbol\Omega$]
\label{thm:dis}
\noindent

\indent
The halting probability of input-free programs, under any probability measure on $\boldsymbol\Sigma$, is less than
$\boldsymbol\Omega$.
\end{theorem}
\begin{proof}
\noindent

\indent
For every positive integer $\ell$, let $\boldsymbol\pi_\ell$ be the probability of an(y) element of $\boldsymbol\Sigma$ with length $\ell$. Therefore, $\sum_{\ell =1}^{\infty}2^\ell\boldsymbol\pi_\ell=1$, since there are $2^\ell$ binary strings of length $\ell$. The halting probability (with the probability measure $\boldsymbol\pi$) is then $\sum_{\ell=1}^{\infty}\mathcal{N}(\ell)
\boldsymbol\pi_{\ell}$; see Definition~\ref{def:Nl}. Let $2^m\boldsymbol\pi_m$ be the maximum of $\{2^\ell\boldsymbol\pi_\ell\}_{\ell=1}^{\infty}$ (which exists by the convergence of $\sum_{\ell =1}^{\infty}2^\ell\boldsymbol\pi_\ell$ and so $\lim_{\ell\rightarrow\infty}2^\ell\boldsymbol\pi_\ell=0$).
We distinguish two cases:

(1) If $2^m\boldsymbol\pi_m=1$, then for every $\ell\neq m$, we should have $\boldsymbol\pi_\ell=0$. Hence,

$\sum_{\ell=1}^{\infty}\mathcal{N}(\ell)
\boldsymbol\pi_{\ell}=\mathcal{N}(m)
\boldsymbol\pi_{m}=\mathcal{N}(m)2^{-m}<
\sum_{\ell=1}^{\infty}\mathcal{N}(\ell)
2^{-\ell}=
\boldsymbol\Omega$,

\noindent
 since there exists some $\ell\neq m$ with $\mathcal{N}(\ell)>0$.

 (2)
So, we can assume that $2^m\boldsymbol\pi_m<1$. In this case,

 $\sum_{\ell=1}^{\infty}\mathcal{N}(\ell)
\boldsymbol\pi_{\ell}\!=\!\sum_{\ell=1}^{\infty}
\mathcal{N}(\ell)2^{-\ell}\!\cdot\!2^\ell
\boldsymbol\pi_{\ell}\!\leqslant\!2^m
\boldsymbol\pi_m\sum_{\ell=1}^{\infty}
\mathcal{N}(\ell)2^{-\ell}\!<\!\sum_{\ell=1}^{\infty}
\mathcal{N}(\ell)2^{-\ell}\!=\!\boldsymbol\Omega$.

Therefore, regardless of the probability measure ($\boldsymbol\pi$), the number $\boldsymbol\Omega$ exceeds the probability of obtaining an input-free halting program by tossing a fair coin a randomly finite number of times.
\end{proof}

 Thus, there is no reason to believe that   the halting probability (of ``all programs of arbitrary size'') is  $\sum_{p\textrm{ halts}}2^{-|p|}(\!=\!\boldsymbol\Omega)$. As pointed out by Chaitin, the series $\sum_{p\textrm{ halts}}2^{-|p|}$ could be greater than $1$, or may even diverge, if the set of programs is not taken to be prefix-free (what ``took ten years until [he] got it right''). So, the fact that, for prefix-free   programs, the real number   $\sum_{p\textrm{ halts}}2^{-|p|}$ lies  between $0$ and $1$ (by Kraft's inequality, Proposition~\ref{prop:kraft}) does not make it a probability of finite strings.

%
%
%
%
%
%
%

Let $n$ be a sufficiently large natural number such that there are $i\neq j\leqslant n$ with $\mathcal{N}(i),\mathcal{N}(j)\!>\!0$. Let $\boldsymbol\Sigma_{\leqslant n}$ be the set of all finite binary strings with length $\leqslant n$. By replacing $\boldsymbol\Sigma$ with $\boldsymbol\Sigma_{\leqslant n}$ and $\infty$ with $n$ in the proof of Theorem~\ref{thm:dis}, one can prove the following theorem, noting that $\ell$'s should be $\leqslant n$ and that $$\sum^{|p|\leqslant n}_{p\textrm{ halts}}2^{-|p|}=
\sum_{\ell=1}^{n}\mathcal{N}(\ell)2^{-\ell}.$$

\begin{theorem}[halting probability of i.f. prog.'s with a bounded length]
\label{thm:disagree-n}
\noindent

\indent
For sufficiently large $n$'s, the halting probability of all the input-free programs with  length $\leqslant n$, under any probability measure on $\boldsymbol\Sigma_{\leqslant n}$, is less than  $\sum^{|p|\leqslant n}_{p\textrm{ halts}}2^{-|p|}.$
\hfill{\mbox{\ding{113}}}
\end{theorem}

\subsection{Halting Probability for  Real Numbers}
Let us see now one of the most recent explanations as to why
$\boldsymbol\Omega$ is considered to be the halting probability of  input-free programs (\cite[p.~1613]{BCH21}):
\begin{itemize}
\item[]\begin{quote}
{
Given a prefix-free machine $M$, one can consider the `halting probability' of $M$, defined by
$\Omega_M\!=\!\sum_{M(\sigma)\downarrow}
2^{-|\sigma|}.$
The term `halting probability' is justified by the following observation: a prefix-free
machine $M$ can be naturally extended to a partial functional from $2^\omega$, the set of
infinite binary sequences, to $2^{<\omega}$, where for $X\in 2^\omega$, $M(X)$ is defined to be $M(\sigma)$ if
some $\sigma\in {\rm dom}(M)$ is a prefix of $X$, and $M(X)\uparrow$ otherwise. The prefix-freeness of $M$
[...]
ensures that this extension is well-defined. With this point of view,
$\Omega_M$ is simply $\mu\{X \in 2^\omega: M(X) \downarrow\}$, where $\mu$ is the uniform probability measure
(a.k.a. Lebesgue measure) on $2^\omega$, that is, the measure where each bit of $X$ is equal
to 0 with probability $1/2$ independently of all other bits.}
\end{quote}
\end{itemize}

See \cite[p.~207]{RS07}
 for a similar explanation.
So,  the expression  ``halting probability'' refers to the probability of some real numbers, not of finite binary strings.
Let us consider   a randomly given real number $\alpha\!\in\!(0,1]$. The probability that $\alpha$ is less than $\frac{1}{4}$ is, of course, $\frac{1}{4}$, since the length of $(0,\frac{1}{4})$ is $\frac{1}{4}$. The probability that $\alpha$ is rational is 0. Let us calculate the probability that the finite string $\texttt{01001}$ is a prefix of the unique infinite binary expansion  after $\texttt{0}\centerdot$ of $\alpha$ (see Definition~\ref{def:base2}). If $\alpha$ is like that, then $\alpha\!=\!(\!\!(\texttt{0}\centerdot \texttt{01001}x_1x_2x_3\dots)\!\!)_{\bf 2}$ for some bits $x_1,x_2,x_3,\cdots$. This means that $\alpha$ belongs to the interval $\mathbb{I}_{\{01001\}}$ (see Definition~\ref{def:iell}), so the probability is $\frac{1}{32}$ (see Example~\ref{ex:iell}).

\begin{lemma}[probability of some events on real numbers]\label{lem:prob}
\noindent

\indent
(1) The probability that a randomly given real $\alpha\!\in\!(0,1]$ has a fixed finite binary string $\sigma$ as a prefix  in its infinite binary expansion after $0\centerdot$ is $\mathfrak{L}(\mathbb{I}_\sigma)$.

(2) The probability that a randomly given real $\alpha\!\in\!(0,1]$ has a prefix from a fixed set of   finite binary strings $S\!\subseteq\!\boldsymbol\Sigma$    in its infinite binary expansion after $0\centerdot$ is $\mathfrak{L}(\bigcup_{\sigma\in S}\mathbb{I}_\sigma)$.
\end{lemma}
\begin{proof}
\noindent

\indent
(1) Every such $\alpha$ belongs to the interval
$\mathbb{I}_\sigma$ (see Definition~\ref{def:iell}). So, the   probability is $\mathfrak{L}(\mathbb{I}_\sigma)$; cf. \cite{Linde24}. Item
(2) follows similarly.
\end{proof}

\begin{corollary}[Omega numbers as probabilities of real numbers]
\label{cor:prob}
\noindent

\indent
(1) The probability that a random  real $\alpha\!\in\!(0,1]$ has a prefix from a fixed prefix-free set of   finite binary strings $S\!\subseteq\!\boldsymbol\Sigma$    in its infinite binary expansion after $0\centerdot$ is ${\varOmega}_S$.

(2) Chaitin's $\boldsymbol\Omega$ is the probability that the unique infinite binary expansion after $0\centerdot$ of a randomly given real $\alpha\!\in\!(0,1]$  contains a finite binary strings as a  prefix   that is the binary code of a halting  input-free program.
\end{corollary}
\begin{proof}
\noindent

\indent
(1) follows from Lemma~\ref{lem:prob}(2)  and Lemma~\ref{lem:big}(4). Item
(2) is a special case of (1) when $S\!=\!\mathbb{H}$ (see Definition~\ref{def:hp}).
\end{proof}

After all, $\boldsymbol\Omega$ is the probability of something, an event on real numbers.

\section{A Short History and Some Suggestions}

\begin{flushright}{\small
\textsl{\texttt{Mathematics is the science of learning how {\em not} to compute}}.
\\
    ---{\em Heinrich Maschke}  (1853--1908); see \cite[p.~667]{Glazier16}. }
\end{flushright}

\medskip

There are very many papers and some books on the so-called {\em Halting Probability} $\boldsymbol\Omega$, also known as Chaitin's Number. We wanted to see if $\boldsymbol\Omega$ defines a probability for the halting problem. And if so, on what measure? What is the distribution of that probability? What is the sample space?
We hopefully gave a systematic understanding of this number and will suggest (\ref{subs:sug} below) some measures based on which a halting probability can be defined, with all the glory of mathematical rigor. Let us observe again that a real number cannot be called a probability if it is just between $0$ and $1$; there should be a measure and a space for a probability that satisfies Kolmogorov axioms (see, e.g.,  \cite{Linde24}): that $\mu(\mathbb{S})\!=\!1$ and $\mu(\bigcup_iS_i)\!=\!\sum_i\mu(S_i)$, where $\mathbb{S}$ is the sample space, $\{S_i\}$ is an arbitrary indexed family  of pairwise disjoint  subsets of $\mathbb{S}$, and the {\em partial} function  $\mu\colon\mathscr{P}(\mathbb{S})\rightarrow[0,1]$ is the probability measure (defined on the so-called {\em measurable} subsets of $\mathbb{S}$).

The number $\boldsymbol\Omega$ was introduced by Chaitin \cite[p.~337]{Chaitin75} in 1975,
when it was denoted by $\omega$. The symbol $\Omega$ appears in Chaitin's second {\em Scientific American} paper \cite{Chaitin88}, where it was defined as the probability that ``a completely random program will halt'' (p.~80).
This is sometimes called ``the secret number,'' ``the magic number,'' ``the number of wisdom,'' etc. \cite[p.~178]{RS07}; it is also claimed ``to hold the mysteries of the universe'' \cite{Gardner79}.
It was stated in \cite{BFGM06}  that  ``The first example of a random real was Chaitin's $\Omega$'' (p.~1411); but as Barmpalias put it in \cite[p.~180, fn.~9]{Barmpalias20}, ``Before Chaitin's discovery, the most concrete Martin-L\"of random real known was a $2$-quantifier definable number exhibited [by] Zvonkin and Levin'' (in 1970).

\subsection{\bf Some Suggestions}\label{subs:sug}

\begin{definition}[integer code, $\mathcal{H}$]\label{def:int}
\noindent

\indent
Every finite  binary string $\sigma\!\in\!\boldsymbol\Sigma$ has an integer code defined as $(\!\!(\texttt{1}\sigma)\!\!)_{\bf 2}-1$, illustrated as follows.
\begin{center}
 \begin{tabular}{c||c|c|c|c|c|c|c|c|c|c|c}
\textrm{binary string} & $\texttt{0}$ & $\texttt{1}$ & $\texttt{00}$ & $\texttt{01}$ & $\texttt{10}$ & $\texttt{11}$ & $\texttt{000}$ & $\texttt{001}$ & $\texttt{010}$ & $\texttt{011}$  & $\!\!\!\!\centerdot\centerdot\centerdot\!\!\!$ \\
  \hline
 \textrm{integer code} & $1$ & $2$ & $3$ & $4$ & $5$ & $6$ & $7$ & $8$ & $9$ & $10$ &  $\!\boldsymbol\cdots\!\!$
\end{tabular}
\end{center}

\noindent
Let $\mathcal{H}$ be the set of the integer codes of all the strings in $\mathbb{H}$ (see Definition~\ref{def:hp}).
\hfill\ding{71}\end{definition}

\begin{example}[integer code]\label{ex:int}
\noindent

\indent
 The integer code of the binary string $\texttt{01001}$ is  40, and the finite binary string with the integer code 25 is $\texttt{1010}$ (see Example~\ref{ex:base2}).
\hfill\ding{71}\end{example}

The  $\boldsymbol\Omega$ number has many interesting properties that have attracted the attention of the brightest minds and made them publish papers in the most prestigious journals and collection books.
Most properties of $\boldsymbol\Omega$, which we
proved not to be
a probability of random strings, are also possessed by  $K\!=\!\sum_{n\in\mathcal{H}}2^{-n}$ (see \cite[p.~33]{Gardner79}).
This number is in the interval $(0,1)$, so it can be a halting probability with a good measure: for a set of positive integers $S\!\subseteq\!\mathbb{N}^+$,   let $\mathfrak{p}(S)\!=\!\sum_{n\in S}2^{-n}$. Then all the probability axioms are satisfied: $\mathfrak{p}(\mathbb{N}^+)\!=\!1$ and $\mathfrak{p}(\bigcup_i\!S_i)\!=\!
\sum_i\mathfrak{p}(S_i)$ for every pairwise disjoint $\{S_i\!\subseteq\!\mathbb{N}^+\}_i$. One question now is: why not take this number as {\em a halting probability}? Notice that this has some non-intuitive properties: if $E$ is the set of all the even positive integers and $O$ is the set of all the odd positive integers, then {\em the probability} that a binary string has an {\em even} integer code becomes  $\mathfrak{p}(E)\!=\!\sum_{n\in E}2^{-n}\!=\!\frac{1}{3}$, and {\em the probability} that a binary string has an {\em odd} integer code turns out to be  $\mathfrak{p}(O)\!=\!\sum_{n\in O}2^{-n}\!=\!\frac{2}{3}$, twice the evenness probability!

For $\boldsymbol\Omega$,  the geometric distribution (see, e.g., \cite{Linde24}) is in play, with the parameter $p\!=\!\frac{1}{2}$. Why not take other parameters, such as $p\!=\!\frac{1}{3}$  and then define a halting probability as $\sum_{\sigma\in\mathbb{H}} 3^{-|\sigma|}$
(or $\sum_{n\in\mathcal{H}}2\cdot 3^{-n}$)? Note that $\sum_{n>0}2\cdot 3^{-n}\!=\!1$, and   Kraft's inequality
 applies here too: $\sum_{\sigma\in S}3^{-|\sigma|}\!\leqslant\!1$ for every prefix-free set $S\subseteq\boldsymbol\Sigma$.  Or, why not Poisson's distribution (see, e.g., \cite{Linde24}) with a parameter $\lambda$? Then,  a halting probability could be   $\sum_{n\in\mathcal{H}}
\frac{e^{-\lambda}\lambda^{-n}}{n!}$.
One key relation in defining  $K$  is the elementary formula $\sum_{n>0}2^{-n}\!=\!1$. Let $\{\alpha_n\}_{n>0}$ be any sequence of positive real numbers such that $\sum_{n>0}\alpha_n\!=\!1$. Then one can {\em define} a halting probability as $\sum_{n\in\mathcal{H}}\alpha_n$ or
$\sum_{\sigma\in\mathbb{H}}2^{-|\sigma|}\alpha_{|\sigma|}$. Most, if not all, of the properties of $\boldsymbol\Omega$ should be possessed by these new {\em probabilities}.
 This seems like a wild, open area to explore.

\paragraph{The Conclusion:}
The number $\boldsymbol\Omega$  {\em is  not} the probability that a randomly given finite binary string is the binary  code of a halting  input-free program under {\em any} probability measure. It {\em is} the probability that   the unique infinite binary expansion after $0\centerdot$  of a randomly given real number in the unit interval has a prefix that is the binary   code of a halting input-free program.
There is no unique {\em halting probability} of finite binary strings, and  one can get different values for it by different probability measures (over a fixed prefix-free programming language).
The following table summarizes our observations about $\boldsymbol\Omega$ and its approximations:

\begin{table}[h]
  \centering
  \begin{tabular}{ c c l }
    $\sum_{p\textrm{ halts}, |p|=\ell }\,2^{-|p|}$ & $\boldsymbol=$ & the probability of getting a halting input-free program   \\
   & & after  tossing a fair coin for $\ell$ times \\
        \hline
      $\sum_{p\textrm{ halts}, |p|\leqslant n}\,2^{-|p|}$ & $\boldsymbol>$ &
    the probability of getting  a halting input-free program   \\
   & & after tossing  a fair coin for some\,$\leqslant\!\,\!n$ times
   (large $n$)
   \\
    \hline
   $\boldsymbol\Omega = \sum_{p\textrm{ halts}}\,2^{-|p|}$ & $\boldsymbol>$ & the probability of getting  a halting input-free program   \\
   & & after tossing a fair coin for {\sl randomly finite} times \\
  \end{tabular}
  \label{table}
\end{table}


\printthanks


\end{document}